\documentclass{amsart}
\usepackage{amsmath,amstext,amssymb,amsfonts,amscd,graphicx,cases}
\usepackage{mathrsfs,accents,mathtools,amsthm,bbm}
\usepackage{xcolor}
\newtheorem{thm}{Theorem}[section]
\newtheorem{cor}[thm]{Corollary}
\newtheorem{lem}[thm]{Lemma}
\newtheorem{prop}[thm]{Proposition}
\newtheorem*{prop*}{Proposition 2.7}
\theoremstyle{definition}
\newtheorem{defn}[thm]{Definition}
\theoremstyle{remark}
\newtheorem{rem}[thm]{Remark}
\numberwithin{equation}{section}

\newcommand{\lr}[1]{\langle #1 \rangle}
\newcommand{\mt}[1]{\mathfrak #1 }

\newcommand{\R}{\mathbb{R}}

\newcommand{\ellg}{\ell_{\gamma}}

\newcommand{\mbq}{\mathbbm{q}}
\newcommand{\mbr}{\mathbbm{r}}
\newcommand{\mbp}{\mathbbm{p}}
\newcommand{\tmbq}{\tilde{\mathbbm{q}}}
\newcommand{\tmbr}{\tilde{\mathbbm{r}}}
\newcommand{\tmbp}{\tilde{\mathbbm{p}}}
\newcommand{\msq}{\mathbbm{q}_{s}}
\newcommand{\msr}{\mathbbm{r}_{s}}
\newcommand{\msp}{\mathbbm{p}_{s}}

\title[Boltzmann equation with small initial data]
{On the Cauchy problem for the cutoff Boltzmann equation with small initial data}

\author{Ling-Bing He}
\address{Department of Mathematical Sciences, TsingHua University, Beijing, 100084, P.R.China}
\email{hlb@tsinghua.edu.cn}

\author{Jin-Cheng Jiang}
\address{Department of Mathematics, National Tsing Hua University, Hsinchu, Taiwan 30013, R.O.C}
\email{jcjiang@math.nthu.edu.tw}

\begin{document}

\begin{abstract}
We prove the global existence of the non-negative unique mild solution for the Cauchy problem 
of the cutoff Boltzmann equation for soft potential model $-1\leq \gamma< 0$  with the small initial data in  three dimensional 
space.  Thus our result fixes  the gap for the case $\gamma=-1$ in three dimensional space in the authors' previous 
work~\cite{HJ17} where the estimate for the loss term was  improperly used.  
The other gap in~\cite{HJ17} for the case $\gamma=0$ in two dimensional space is recently fixed by 
Chen, Denlinger and Pavlovi\'{c}~\cite{CDP21}.
The initial data $f_{0}$ is non-negative and  satisfies that  $\|\lr{v}^{\ellg} f_{0}(x,v)\|_{L^{3}_{x,v}}\ll1$  and $\|\lr{v}^{\ellg} f_0\|_{L^{15/8}_{x,v}}<\infty$  where $\ellg=0$ when 
$\gamma=-1$ and $\ellg=(1+\gamma)^{+}$  when $-1<\gamma<0$.  
We also show that the solution scatters with respect to the kinetic transport operator. 
The novel contribution of this work lies in the exploration of the symmetric property of the gain term  
 in terms of weighted estimate. It  is the key ingredient for solving the  model  
 $-1<\gamma<0$ when applying the Strichartz estimates.
\end{abstract}

\maketitle

\section{Introduction}

We consider the Cauchy problem for the cutoff Boltzmann equation 
\begin{equation}\label{E:Cauchy}
\left\{
\begin{aligned}
&{\partial_t f}+v\cdot\nabla_x f =Q(f,f)\\
&f(0,x,v)=f_0(x,v)
\end{aligned}
\right.
\end{equation}
in $(0,\infty)\times\mathbb{R}^N\times\mathbb{R}^N, N=2,3,$ where the initial data is small in $L^{N}_{x,v}$ space.
Recall that the collision operator $Q(f,f)$ is given by 
\[
Q(f,f)(v)=\int_{\mathbb{R}^N}\int_{\omega\in S^{N-1}}(f'f'_{*}-ff_{*}) B(v-v_{*},\omega)\; d\omega dv_{*},
\]
and $d\omega$ is the solid element in the direction of unit vector $\omega$.
We use the abbreviations $f'=f(x,v',t)\;,\;f'_{*}=f(x,v'_{*},t)\;,\;f_{*}=f(x,v_{*},t)$, 
and the relation between the pre-collisional velocities of particles and
after collision is given by
\begin{equation}\label{E:pre-collision}
v'=v-[\omega\cdot(v-v_{*})]\omega\;,\;v'_{*}=v_{*}+[\omega\cdot(v-v_{*})]\omega\;,\;\omega\in S^{N-1}.
\end{equation}

In this paper, we consider the cutoff soft potential model, i.e., the collision kernel $B$  being the product of 
kinetic part and angular part, 
\begin{equation}\label{D:kernel}
B(v-v_*,\omega)=|v-v_*|^{\gamma}\;b(\cos\theta),\;0\leq\theta\leq \pi/2\;,
\end{equation}
where 
\[
-N<\gamma<0\;,\;\cos\theta=\lr{\omega,(v-v_*)/|v-v_*|}\;,
\]
and the angular function $b(\cos\theta)$ satisfies the Grad's cutoff assumption 
\begin{equation}\label{D:Grad}
\int_{S^{N-1}} b(\cos\theta)d\omega<\infty.
\end{equation}
When $\gamma=0$, the kernel~\eqref{D:kernel} is called the Maxwell molecules. When the cutoff condition~\eqref{D:Grad}
is satisfied, the collision operator $Q$ can be split into the gain term $Q^{+}$ and the loss term $Q^{-}$.    
It is useful to introduce the bilinear gain term
\[
Q^+(f,g)(v)=\int_{{\R}^N}\int_{S^{N-1}}  f(v')g(v_*') B(v-v_*,\omega)
d\omega dv_*,
\] 
and the bilinear loss term
\[
Q^{-}(f,g)(v)=\int_{{\R}^N}\int_{S^{N-1}}  f(v)g(v_*) B(v-v_*,\omega)
d\omega dv_*.
\]

\subsection{Short review} Let us briefly recall the progress on the Cauchy problem ~\eqref{E:Cauchy} for cutoff model with small initial data. 
To the best of our knowledge, Illner and Shinbrot~\cite{IS84} first showed the global existence of solutions for 
the Cauchy problem~\eqref{E:Cauchy} for several cutoff models when the initial data has exponential decay 
in spatial variable and has suitable weight in velocity variable.  They also discussed the asymptotic behavior of the solutions. 
The iteration method in~\cite{IS84}  comes from the earlier work~\cite{KS78} of Kaniel and Shinbrot  who designed it for 
the study of the initial boundary value problem in a bounded domain. Now it is  
called the Kaniel-Shinbrot iteration.   
Various of results about the   small initial data Cauchy problem ~\eqref{E:Cauchy} for different models 
were then obtained during 
that decade by many authors through the same iteration or fixed point argument, see~\cite{BPT88, G96} and reference 
therein for more details. Please note that  the assumption that the initial data has exponential 
decay in  spatial variable or in velocity variable is necessary to get these results.  

With the first appearance of Strichartz estimates for the kinetic equation in the note of Castella and 
Perthame~\cite{CP96}, these  new estimates seem to be a promising tool to solve~\eqref{E:Cauchy} 
with initial data being small in Lebesgue space instead of decaying exponentially.
Indeed, with initial data small in Lebesgue assumption, Bournaveas et al.~\cite{BCGP08} used it to prove the existence 
of global weak solution 
for a nonlinear kinetic system modeling chemotaxis. For the Boltzmann 
equation, Ars\'{e}nio~\cite{Ars11} considered a non-conventional collision kernel whose kinetic part is 
$L^p$ integrable for some $p$ depending on dimension, and then proved the existence of global weak solution
for small data. But the uniqueness of the solution is unknown. The reason that the Strichartz estimates is not handy as one expects in 
solving small initial data Boltzmann equation lies in the fact that the loss term does not enjoy the same symmetry as the gain 
term does in the Lebesgue space. More precisely, the gain and loss terms both satisfy 
\begin{equation}
\|Q^{\pm}(f,g)\|_{L^{\mt{r}}_{v}(\mathbb{R}^N)}\leq C\|f\|_{L^{\mt{p}}_{v}(\mathbb{R}^N)}\|g\|_{L^{\mt{q}}_{v}(\mathbb{R}^N)},
\end{equation}
where the norm is taking on the velocity variable and the exponents $\mt{p},\mt{q},\mt{r}$ satisfying the scaling condition
\begin{equation}
1/\mt{p}+1/\mt{q}=1+\gamma/N+1/\mt{r},
\end{equation}
while the estimate for the loss term requires additional condition, 
\begin{equation}\label{non-symmetry}
1/\mt{p}<1/\mt{r}, 
\end{equation}
which means that $f$ and $g$ need to be treated differently when dealing with the loss term.  
In the authors' previous work~\cite{HJ17}, the constraint~\eqref{non-symmetry}, included in the proof but 
not the statement of Lemma 2.4 there, was neglected when applying the Strichartz estimates to solve the Cauchy problem.  
Therefore the result in~\cite{HJ17}, c.f. Proposition~\ref{Thm-Gain},  holds for the gain term only 
Boltzmann equation instead of full equation (see also the paragraph before and after Lemma~\ref{T:Loss} below).
 The work ~\cite{HJ17} pointed out that when the exponent
of kinetic part equals $\gamma=2-N$, we can find suitable Strichartz spaces, c.f.~\eqref{solvable-g} , 
where the global mild solution  for gain term only Boltzmann equation exists if the initial data $f_0$ is small in 
$L^N_{x,v}$.  

Very recently, Chen, Denlinger and 
Pavlovi\'{c}~\cite{CDP21} studied the full Boltzmann equation for the case $N=2$, Maxwell molecules, by a 
 different approach. 
Their idea can be sketched as  follows. 
Using the fact that the kinetic transport equation can be converted to the free Schr\"{o}dinger equation by Wigner transform 
and vice verse by inverse Wigner transform, they proved the spacetime estimates for the nonlinear 
Schr\"{o}dinger equation to conclude the existence of mild solution for  gain term only Boltzmann equation 
when the initial data is small enough in $L^{2}_{x,v}$.  Due to the fact that kinetic transport operator, weight in velocity
and differential operator for spatial variable are commuting with each other,  they showed that if $f_0$ is small in 
$L^2_{x,v}$ and additionally $\| \lr{v}^{\frac{1}{2}+}\lr{\nabla_x}^{\frac{1}{2}+}f_0\|_{L^2_{x,v}}$ is finite, then  
the quantity $\|\lr{v}^{\frac{1}{2}+}\lr{\nabla_x}^{\frac{1}{2}+}f_{+}\|_{L^{\infty}_t L^2_{x,v}}$  remains finite where $f_{+}$
denotes the solution of the gain term only Boltzmann equation. The propagation of regularity and moment  
thus ensures the loss term is well-defined when plugging in $f_{+}$. Note  
the fact that if the initial $f_0$ is non-negative, then solution of the gain term only Boltzmann equation is also non-negative. 
Combining all the facts, using the solution  of gain term only Boltzmann equation as an upper bound of the 
``beginning condition'' of the Kaniel-Shinbrot iteration, one can ensure the global existence of the mild solution for full 
equation.  The uniqueness of the solution is not provided by Kaniel-Shinbrot iteration. Fortunately it can be saved by 
the fact that the solution of  the gain term only Boltzmann equation is an upper bound of that for the full equation and the former
lies in solution space already.  The scattering and propagation of moment and regularity for the solution of the full equation
are also proved by the similar idea. 

\subsection{Main results} The main purpose of this paper is to solve the problem~\eqref{E:Cauchy} for the case $N=3$ when 
$\gamma$ satisfying $-1\leq \gamma< 0$. Instead of using
the correspondence between the Sch\"{o}rdinger equation and kinetic transport equation, the result of~\cite{HJ17} 
for the gain term only Boltzmann equation will be our starting point, c.f. Proposition~\ref{Thm-Gain}. 
On the other hand, we should adopt the strategy of~\cite{CDP21} to recover the solution for the full equation 
from that for the gain  term only Boltzmann equation. To get rid of the fact that the loss term is not symmetric, we 
also need an additional assumption about the initial data besides it is small in velocity-weighted $L^{3}_{x,v}$ 
(no weighted when $\gamma=-1$). We should assume that  the initial data is also bounded in velocity-weighted 
$L^{15/8}_{x,v}$ (no weighted when $\gamma=-1$).  Here the exponent $15/8$ is just one of possible options. 
 Please note that our additional assumption for the initial 
data does not require the additional regularity in spatial variable nor additional weight in velocity variable.

This difference on assumptions reflects the difference of the method. It is 
interesting and worth to explore more about this. First we note that the exponent $\gamma=2-N$ of kinetic part of the collision 
kernel is special in the sense that it is scaling critical case in the content of dispersive equation. The spacetime 
$L^{2}_{x,v}$ estimates in~\cite{CDP21} is non-trivial since it is an end point estimate. Comparing $L^{2}_{x,v}$ space with 
the $L^{\mbr}_{x}L^{\mbp}_{x},\; {\mbr}<N=2, {\mbp}>2$ in Proposition~\ref{Thm-Gain}, it is not surprising that   
to recover the solution for full equation from that of gain term only equation in the $L^{2}_{x,v}$ space, one needs 
to require the initial data has additional regularity in spatial variable as well as additional weight in velocity variable.
If one follows the approach of~\cite{CDP21} to study the case $N=3$, both requirements seem to be unavoided again.
Also the other difficulty that will encounter is that the exponent of kinetic part of the collision kernel under consideration 
is $-1\leq \gamma<0$ when $N=3$ which is unlike
$\gamma=0$ when $N=2$ as the latter is more convenient when applying the Fourier transform to the gain term of 
the collision operator.     
 
The novel contribution of this work is that the symmetric property of the gain term is explored further in terms of 
weighted estimate and this is the key step to study the model $-1<\gamma<0$ 
when applying the Strichartz estimate to solve the porblem.
 The gain term enjoys  two different estimates based on two different scaling relations, i.e.,   
\[
\|\lr{v}^{\ell}Q^{+}(f,g)\|_{L^{\mt{r}}_{v}(\mathbb{R}^N)}\leq C\|\lr{v}^{\ell} f\|_{L^{\mt{p}}_{v}(\mathbb{R}^N)}
\|\lr{v}^{\ell}g\|_{L^{\mt{q}}_{v}(\mathbb{R}^N)}, 
\]
where $1/\mt{p}+1/\mt{q}=1+\gamma/N+1/\mt{r},\;\ell\geq 0$, and 
\[
\|\lr{v}^{\ell} Q^{+}(f,g)\|_{L^{\mt{r}_{m}}}\leq C(\mt{p}_{m},\ell) \|\lr{v}^{\ell} f\|_{L^{\mt{p}_{m}} }\|\lr{v}^{\ell} g\|_{L^{\mt{q}_{m}}}
\]
where ${1}/{\mt{p}_{m}}+{1}/{\mt{q}_{m}}+{1}/{m}=1+{\gamma}/{N}+{1}/{\mt{r}_{m}},\;\ell>N/m$.
On the other hand, the loss term only satisfies the second estimate above while the constraint~\eqref{non-symmetry} is unchanged. Please see  Proposition~\ref{P:Convolution-w} and Proposition~\ref{P:Loss} for more details. 
This new discovery on the property of gain term allows us to solve the Cauchy problem~\eqref{E:Cauchy}  
for the soft potential model with exponent $-1<\gamma<0$ beyond $\gamma=-1$.   
 It seems to us that this approach is more straight forward for the case $N=3$, thus our argument is shorter than 
 that in~\cite{CDP21}.  Unfortunately this method does not work for the case $N=2$ since $\gamma<0$ is below
the critical case $\gamma=0$ while $-1<\gamma<0$ is above $\gamma=-1$ for $N=3$.

To state the main results, let us introduce the mixed Lebesgue norm
\[
\|f(t,x,v)\|_{L^q_tL^r_xL^p_v}
\] where the notation  $L^q_tL^r_xL^p_v$ stands for the space $L^q(\mathbb{R};L^r
(\mathbb{R}^N;L^p(\mathbb{R}^N)))$. It is understood that we use 
$L^q_t(\mathbb{R})=L^q_t([0,\infty))$ for the well-posedness problem which can 
be done by imposing support restriction to the inhomogeneous  Strichartz estimates. 
We use $L^a_{x,v}$ to denote $L_x^a(\mathbb{R}^N;L_v^a(\mathbb{R}^N))$.

We also need to give a precise meaning of the scattering of the solution with respect to kinetic transport operator. 
Here we  say that a global solution $f\in C([0,\infty),L^a_{x,v})$ scatters in $L^a_{x,v}$ as 
$t\rightarrow\infty$ if there exits $f_{\infty}\in L^a_{x,v}$ such that 
\begin{equation}
\|f(t)-U(t)f_{\infty}\|_{L^a_{x,v}}\rightarrow 0
\end{equation}
where $U(t)f(x,v)=f(x-vt,v)$ is the solution map of the kinetic transport equation 
\[
\partial_t f+v\cdot\nabla_x f=0. 
\]
Please see the interesting discussion in~\cite{BGGL16} about scattering of the solution and its relation with H-theorem.

For the purpose of clear representation, we should prove first the case $\gamma=-1$ then generalize 
the argument to the case $-1<\gamma<0$. 
First we state the result for the case $\gamma=-1$ as  follows. 
\begin{thm}\label{result1} 
Let $N=3$ and assume the kernel B in~\eqref{D:kernel} has $\gamma=-1$ and satisfies~\eqref{D:Grad}. 
There exists a small number $\eta>0$ such that if the initial data 
$$
f_{0}\in {\mathbb B}_{\eta}=\{f_{0} | f_0\geq 0,\; \|f_{0}\|_{L^{3}_{x,v}}<\eta,\;\|f_0\|_{L^{15/8}_{x,v}}<\infty\}\subset L^3_{x,v},
$$
then the Cauchy problem~\eqref{E:Cauchy} admits a unique and non-negative 
mild solution 
\[
 f\in C([0,\infty),L^3_{x,v})\cap
L^{\mbq}([0,\infty],L^{\mbr}_xL^{\mbp}_v), 
\]
where the triple $(\mbq,\mbr,\mbp)$ lies in the set
\begin{equation}\label{solvable}
\{ (\mbq,\mbr,\mbp) | \;\frac{1}{\mbq}=\frac{3}{\mbp}-1\;,\;\frac{1}{\mbr}=\frac{2}{3}-\frac{1}{\mbp}\;,\;
\frac{1}{3}<\frac{1}{\mbp}<\frac{4}{9}\}. 
\end{equation} 
The solution map $f_0\in {\mathbb B}_{\eta}\rightarrow f\in L^{\mbq}_tL^{\mbr}_xL^{\mbp}_v$
is Lipschitz continuous and the solution $f$ scatters with respect to the kinetic 
transport operator in $L^3_{x,v}$. 
\end{thm}

Next we state the result for the case $-1<\gamma<0$. This part is not studied in the authors' previous work~\cite{HJ17},
even for the gain term only Boltzmann equation.
We use the notation $\ellg^+$ to denote $\ellg+\varepsilon$ where $\varepsilon>0$ is arbitrary small. 
The result is as follows. 
\begin{thm}\label{result2}
Let $N=3$ and assume the kernel B~\eqref{D:kernel} has $-1<\gamma<0$  and satisfies~\eqref{D:Grad}. 
Let $\ellg=(1+\gamma)^+<3/2$. 
There exists a small number $\eta>0$ such that if the initial data 
$$
f_{0}\in {\mathbb B}^{\ellg}_{\eta}=\{f_{0} | f_0\geq 0,\; \| \lr{v}^{\ellg} f_{0}\|_{L^{3}_{x,v}}<\eta,\;
\|\lr{v}^{\ellg}f_0\|_{L^{15/8}_{x,v}}<\infty\}\subset L^3_{x,v},
$$
 then the Cauchy problem~\eqref{E:Cauchy} admits a unique and non-negative 
 mild solution 
\[
\lr{v}^{\ellg} f\in C([0,\infty),L^3_{x,v})\cap
L^{\mbq}([0,\infty],L^{\mbr}_xL^{\mbp}_v), 
\]
where the triple $(\mbq,\mbr,\mbp)$ lies in the set~\eqref{solvable}.
The solution map $\lr{v}^{\ellg}f_0\in {\mathbb B}^{\ellg}_{\eta}\rightarrow \lr{v}^{\ellg} 
f\in L^{\mbq}_tL^{\mbr}_xL^{\mbp}_v$ 
is Lipschitz continuous and the solution $\lr{v}^{\ellg} f$ scatters with respect to the kinetic 
transport operator in $L^3_{x,v}$. 
\end{thm}

Finally, we note that the local  wellposedness result of Theorem 1.3 in~\cite{HJ17} holds for gain term 
only Boltzmann equation instead of full equation due to the same reason mentioned above. 
The method of Theorem~\ref{result1} can also fix the problem and we have the following result.

\begin{thm}\label{T3}
Let $N=2$ or $3$ and $B$ defined in~\eqref{D:kernel} satisfies~\eqref{D:Grad}  and 
$-N<\gamma<2-N$.  The Cauchy problem~\eqref{E:Cauchy} is locally wellposed
when the initial data lies in 
\[
B_{R}=\{f_0\in L^{a_{s}}_{x,v}
(\mathbb{R}^N\times\mathbb{R}^N): f_{0}\geq 0,\; \|f_0\|_{L^{a_{s}}_{x,v}}<R\;,\;\|f_0\|_{L^{m}_{x,v}}<\infty\}
\subset L^{a_{s}}_{x,v}
\]
where $a_{s}={2N}/({\gamma+N}),\;m=2N/[(\gamma+N)(5\alpha-1)]$ and 
${1}/{2}<\alpha<(N+1)/(2N)$.
More specially, for any $R>0$ there exists a $T=T(\msr,\msp,R)$ such that for all $f_0\in B_{R}$, the Cauchy problem~\eqref{E:Cauchy} admits a unique mild solution 
\[
f\in C([0,T),L^{a_{s}}_{x,v})\cap
L^{\msq}([0,T],L^{\msr}_xL^{\msp}_v),
\]
where the triple $(\msq,\msr,\msp)$ lies in the set
\begin{equation}\label{solvable-2}
\begin{split}
\{ (\frac{1}{\msq},\frac{1}{\msr},\frac{1}{\msp})  |  &\frac{1}{\msq}=\frac{(2\alpha-1)(\gamma+N)}{2} ,\\
&{\hskip 1cm}\;\frac{1}{\msr}=\frac{(1-\alpha)(\gamma+N)}{N},
\frac{1}{\msp}=\frac{\alpha(\gamma+N)}{N}\}.
\end{split}
\end{equation} 
The solution map $f_0\in B_{R}\rightarrow f\in L^{\msq}([0,T];L^{\msr}_xL^{\msp}_v)$ 
is Lipschitz continuous.
\end{thm}

\subsection{Organization of the paper} The proof of Theorem~\ref{result1} is lengthy as it contains many parts. We organize the paper as follows. 

In Section 2 we 
prove the  global existence of solutions $f_{+}$ for the gain term only Boltzmann equation with small initial data 
in the suitable Strichartz spaces. In particular, the solution is non-negative if the initial data is non-negative. This part is mainly the 
reminiscence of~\cite{HJ17}. Two useful estimates induced by the condition $\|f_0\|_{L^{15/8}_{x,v}}<\infty$  are also included. 

In Section 3 we use $h_{1}=0$ and $g_{1}=f_{+}\geq 0$ as the lower and upper bounds to build the beginning condition,
 $0\leq h_{1}\leq h_{2}\leq g_{2}\leq g_{1}$, for the Kaniel-Shinbrot iteration. 
 With the aid of $\|  f_0\|_{L^{15/8}_{x,v}}<\infty$, the solution $f_{+}$  of Section 2  ensures the lose term
  $Q^{-}(h_{2},g_{1})$ is well-defined in the sense that it lies in a suitable Strichartz space also.
  The same trick also makes sure that each term in the iteration process is well-defined, thus we can
 run the Kaniel-Shinbro iteration to get the lower and upper solutions $h$ and $g$ of system~\eqref{Iteration-limit}.    
 To close the iteration, we need to show $g=h$. The argument requires again that the assumption that 
 $\| f_0\|_{L^{15/8}_{x,v}}<\infty$.
 
 To check the uniqueness of the solution,
 we  consider the difference of the solutions for 
the full equation and the corresponding difference equation with zero initial data.  
The non-negativity of the solution helps us when using the continuity argument. The continuity in time, scattering of 
the solution and the solution map is Lipschitz continuous can be shown by the standard argument.

In the end of Section 3, we include the proof of Theorem~\ref{T3} by pointing out the main difference with that of 
Theorem~\ref{result1}.

The proof of Theorem~\ref{result2} is included in Section 4 which generalizes the argument of 
Theorem~\ref{result1} after building the weighted estimates for the gain and loss terms as 
well as weighted Strichartz estimates.

\section{The gain term only Boltzmann equation}

\subsection{Global Existence for the gain term only Boltzmann equation}
The main result, Proposition~\ref{Thm-Gain}, is indeed included in~\cite{HJ17}. To be self-contained, we will review the main strategy of the proof which 
is needed for the further analysis.   

First  we  recall the  Strichartz estimates for the kinetic transport equation,
\begin{equation}\label{E:KT}
\left\{
\begin{aligned}
&\partial_t u(t,x,v)+v\cdot\nabla_x u(t,x,v) =F(t,x,v),\;\;(t,x,v)\in (0,\infty)\times\mathbb{R}^N\times\mathbb{R}^N,\\
& u(0,x,v)=u_0(x,v).
\end{aligned}
\right.
\end{equation}
To state the Strichartz estimates for~\eqref{E:KT},  we need the following definition.
\begin{defn}\label{D:admissible}
We say that the exponent triplet $(q,r,p)$, for $1\leq p,q,r\leq\infty$ is KT-admissible if 
\begin{equation}
\frac{1}{q}=\frac{N}{2}{\Big ( \frac{1}{p}-\frac{1}{r} }{\Big )}
\end{equation}
\begin{equation}\label{pr-star-2}
1\leq a\leq\infty,\;\;p^*(a)\leq p\leq a, \;\;a\leq r\leq r^*(a)
\end{equation}
except in the case $N=1,\;(q,r,p)=(a,\infty,a/2)$. Here by $a=$HM$(p,r)$ we have denoted the harmonic 
means of the exponents $r$ and $p$, i.e.,
\begin{equation}
\frac{1}{a}=\frac{1}{2}{\Big ( \frac{1}{p}+\frac{1}{r} }{\Big )}
\end{equation}
Furthermore, the exact lower bound $p^*$ to $p$ and the exact upper bound $r^*$ to $r$ are 
\begin{equation}\label{pr-star-1}
\left\{
\begin{array}{lll}
p^*(a)=\frac{Na}{N+1}, & r^*(a)=\frac{Na}{N-1} & {\rm if}\;\frac{N+1}{N}\leq a\leq\infty, \\
p^*(a)=1, & r^*(a)=\frac{a}{2-a} & {\rm if}\; 1\leq a\leq \frac{N+1}{N}.
\end{array}
\right.
\end{equation}
\end{defn}
The triplets of the form $(q,r,p)=(a,r^*(a),p^*(a))$ for 
$\frac{N+1}{N}\leq a<\infty$ are called endpoints.  
The endpoint Strichartz estimate for the kinetic equation is false in all dimensions has been proved recently by 
Bennett, Bez, Guti\'{e}rrez and Lee~\cite{BBGL14}.

The mild solution of the kinetic equation~\eqref{E:KT} can be written as 
\begin{equation}\label{integal-equation}
u=U(t)u_0+W(t)F
\end{equation}
where 
\begin{equation}\label{D:solution-maps}
U(t)u_0=u_0(x-vt,v)\;,\; W(t)F=\int_0^{t} U(t-s)F(s)ds.
\end{equation}
The estimates for the operator $U(t)$ and $W(t)$ respectively in the mixed Lebesgue norm 
$\|\cdot\|_{L^q_tL^r_xL^p_v}$ are called homogeneous and inhomogeneous Strichartz 
estimates respectively. We record the estimates  for the equation~\eqref{integal-equation} 
in the following Proposition where  $p'$ denotes the conjugate exponent of $p$ and so on.
\begin{prop}[\cite{Ovc11},\cite{BBGL14}]\label{Strichartz-est} 
Let $u$ satisfies~\eqref{E:KT}.  The estimate
\begin{equation}\label{E:Strichartz}
\|u\|_{L^q_tL^r_xL^p_v}\leq C(q,r,p,N)( \|u_0\|_{L^a_{x,v}} +\|F\|_{L^{\tilde{q}'}_tL^{\tilde{r}'}_xL^{\tilde{p}'}_v}  )
\end{equation}
holds for all $u_0\in L^a_{t,x}$ and all $F\in {L^{\tilde{q}'}_tL^{\tilde{r}'}_xL^{\tilde{p}'}_v} $ if and only if 
$(q,r,p)$ and $(\tilde{q},\tilde{r},\tilde{p})$ are two KT-admissible exponets triplets and $a=$HM$(p,r)=$HM$(\tilde{p}',\tilde{r}')$
with the exception of $(q,r,p)$ being an endpoint triplet.
\end{prop}

Now we consider the Cauchy problem for the gain term only Boltzmann equation 
\begin{equation}\label{G-Cauchy}
\left\{
\begin{aligned}
&{\partial_t f_+}+v\cdot\nabla_x f_+ =Q^{+}(f_+,f_+)\\
&f_+(0,x,v)=f_0(x,v).
\end{aligned}
\right.
\end{equation}
We define the solution map by
\begin{equation}\label{G-integral-equation}
Sf_+(t,x,v)=U(t)f_0+W(t)Q^{+}(f_+,f_+).
\end{equation}
From~\eqref{G-integral-equation}  and Proposition~\ref{Strichartz-est}, we will see that it holds 
the estimates
\begin{equation}\label{contraction}
\begin{split}
\|Sf_+\|_{L^q_tL^r_xL^p_v}& \leq C_{0}\|f_{0}\|_{L^a_{x,v}}+C_{1}\|Q^+(f_+,f_+)\|^{2}
_{L^{\tilde{q}'}_tL^{\tilde{r}'}_xL^{\tilde{p}'}_v} \\
&  \leq C_{0}\|f_{0}\|_{L^a_{x,v}}+C_{2}\|f_+\|^2_{L^q_tL^r_xL^p_v},
\end{split}
\end{equation}
%\begin{equation}
%\|Sf_+\|_{X}\leq C_{0}\|f_{0}\|_{Y}+C_{1}\|f_+\|^{2}_{X}
%\end{equation}
for suitable Strichartz spaces $L^q_tL^r_xL^p_v$ and $L^{\tilde{q}'}_tL^{\tilde{r}'}_xL^{\tilde{p}'}_v$.
Then the contraction mapping argument will work if the initial data is small in space $L^a_{x,v}$.
The  key lies in the fact that if there exist admissible triplets $(q,r,p)$ and $(\tilde{q}', \tilde{r}',\tilde{p}')$
with $HM(p,r)=HM(\tilde{p}',\tilde{r}')$ such that  the estimate
\begin{equation}\label{desire-est}
\|Q^{+}(f_+,f_+)\|_{L^{\tilde{q}'}_tL^{\tilde{r}'}_xL^{\tilde{p}'}_v} \leq C\|f_+\|^{2}_{L^q_tL^r_xL^p_v}
\end{equation}
holds.

In order to prove the existence of such triplets, we need the estimates for the gain term in $v$ variable. Indeed it is included in 
Theorem 1 and Theorem 2 of~\cite{ACG10} by  Alonso, Carneiro and Gamba.  We collect what we need as follows. 
\begin{prop}[\cite{ACG10}]\label{P:Convolution}
 Let $1<\mt{p}, \mt{q}, \mt{r} <\infty$ with $-N<\gamma\leq 0$ and 
 \begin{equation}\label{scaling-relation-0}
 1/\mt{p}+1/\mt{q}=1+\gamma/N+1/\mt{r}.
 \end{equation}
 Assume the kernel~\eqref{D:kernel}{\rm:} 
\[
B(v-v_*,\omega)=|v-v_*|^{\gamma}b(\cos\theta)
\]
with $b(\cos\theta)$ satisfies Grad's cutoff assumption~\eqref{D:Grad}.
Then the bilinear operator $Q^{+}(f,g)$ is a bounded operator from $L^{\mt{p}}(\mathbb{R}^N)\times 
L^{\mt{q}}(\mathbb{R}^N)\rightarrow L^{\mt{r}}(\mathbb{R}^N)$ via the estimate 
\begin{equation}\label{G-convolution-est}
\|Q^{+}(f,g)\|_{L^{\mt{r}}_{v}(\mathbb{R}^N)}\leq C\|f\|_{L^{\mt{p}}_{v}(\mathbb{R}^N)}\|g\|_{L^{\mt{q}}_{v}(\mathbb{R}^N)}.
\end{equation}
\end{prop}

As we mention in the introduction that the main result of~\cite{HJ17}  actually holds for gain term only Boltzmann 
equation instead of full equation due to the negligence of constrain in the estimate for the loss term. 
More precisely, the proof of ~\cite{HJ17} infers the following result. 
\begin{prop}[cf. Theorem 1.1 in~\cite{HJ17}]\label{Thm-Gain} 
Let $N=2$ or $3$ and $B$ defined in~\eqref{D:kernel} satisfies~\eqref{D:Grad} and $\gamma=2-N$. 
The Cauchy problem~\eqref{G-Cauchy} is 
globally wellposed in $L^N_{x,v}$ when the initial data is small enough. 
More specially, there exists $\eta>0$ small enough  such that for all $f_0$ in the set
\[ B_{\eta}=\{f_0\in L^N_{x,v} (\mathbb{R}^N\times
\mathbb{R}^N): f_0\ge0\;{\rm and}\; \|f_0\|_{L^N_{x,v}}<\eta\}
\] 
there exists a globally unique  mild solution 
\[
f_+\in C([0,\infty),L^N_{x,v})\cap
L^{\mbq}([0,\infty],L^{\mbr}_xL^{\mbp}_v) 
\]
where the triple $(\mbq,\mbr,\mbp)$ lies in the set
\begin{equation}\label{solvable-g}
\{ (\mbq,\mbr,\mbp) | \;\frac{1}{\mbq}=\frac{N}{\mbp}-1\;,\;\frac{1}{\mbr}=\frac{2}{N}-\frac{1}{\mbp}\;,\;
\frac{1}{N}<\frac{1}{\mbp}<\frac{N+1}{N^2}\}. 
\end{equation} 
The solution map $f_0\in B_{\eta}\subset L^N_{x,v}\rightarrow f_+\in L^{\mbq}_tL^{\mbr}_xL^{\mbp}_v$ 
is Lipschitz continuous and the solution $f_+$ scatters with respect to the kinetic 
transport operator in $L^N_{x,v}$.  
\end{prop}

\begin{defn}\label{solvable-tri}
We use the notation $(\mbq,\mbr,\mbp)$  to address that it stems from the usual KT-admissible triplet $(q,r,p)$  and 
lies in the set~\eqref{solvable-g} (i.e., ~\eqref{solvable}). We should call that $(\mbq,\mbr,\mbp)$ is a solvable triplet.
We say that $(\tmbq',\tmbr',\tmbp')$ is the conjugate triple of the solvable triplet $(\mbq,\mbr,\mbp)$ if 
$HM(\mbp,\mbr)=HM(\tmbp',\tmbr')$ and 
\begin{equation}\label{desire-est-sol}
\|Q^{+}(f,f)\|_{L^{\tmbq'}_tL^{\tmbr'}_xL^{\tmbp'}_v} \leq C\|f\|^{2}_{L^{\mbq}_tL^{\mbr}_xL^{\mbp}_v}.
\end{equation}
\end{defn}

Before we consider the full Boltzmann equation, we also need the following result. 
\begin{cor}\label{pos-gain} 
Under the same conditions as Proposition~\ref{Thm-Gain}, if we furthermore  assume 
$f_0\geq 0$, then the solution $f_{+}$ is also non-negative. 
\end{cor}

For the proof of Corollary~\ref{pos-gain} and later analysis, we include  the portion of 
the proof of Proposition~\ref{Thm-Gain} which shows that if the admissible triplets  $(q,r,p)$ lie 
in~\eqref{solvable-g}, there exist corresponding $(\tilde{q}', \tilde{r}',\tilde{p}')$ such 
that~\eqref{desire-est} holds, i.e., we can find solvable triple $(\mbq,\mbr,\mbp)$ and
 its conjugate triplet $(\tmbq,\tmbr,\tmbp)$.

\begin{proof}[Proof of~\eqref{desire-est-sol}]
For $v$ variable, we let $\mt{r}=\tilde{p}',\; \mt{p}=\mt{q}=p$ in~\eqref{G-convolution-est}, thus 
\begin{equation}\label{Condition:v}
\frac{2}{p}=1+\frac{\gamma}{N}+\frac{1}{\tilde{p}'}.
\end{equation}
 For $x$ variables, the condition for being able to apply the H\"{o}lder inequality is
\begin{equation}\label{Condition:x}
2\tilde{r}'=r, \; r\geq 2.
\end{equation}
Furthermore the Strichartz inequality demands the relation of pairs $(p,r),(\tilde{p}',\tilde{r}')$, 
\begin{equation}\label{Condition:har}
\frac{1}{p}+\frac{1}{r}=\frac{1}{\tilde{p}'}+\frac{1}{\tilde{r}'}.
\end{equation}
To apply the H\"{o}lder inequality to $t$ variable, we need
\begin{equation}\label{E:1t-variable}
\frac{2}{q}=\frac{1}{\tilde{q}'}<1,
\end{equation}
that is 
\begin{equation}\label{Condition:14}
\frac{2}{q}+\frac{1}{\tilde{q}}=1\;,\; \frac{1}{q}<\frac{1}{2}.
\end{equation}
Finally the KT-admissible conditions
\begin{align}
&\frac{1}{q}=\frac{N}{2}(\frac{1}{p}-\frac{1}{r})>0,\label{Condition:15}\\
&\frac{1}{\tilde{q}}=\frac{N}{2}(\frac{1}{\tilde{p}}-\frac{1}{\tilde{r}})>0
\label{Condition:16}
\end{align}
must be fulfilled.

We note that once $\gamma,p,r$ are given, $q,\tilde{p},\tilde{r},\tilde{q}$ 
are determined. Thus we rewrite above conditions as      
\begin{subequations}
\begin{align}
&\frac{1}{p}+\frac{1}{r}=1+\frac{\gamma}{N} \qquad {\rm from}\;~\eqref{Condition:v}
\;{\rm and\;}~\eqref{Condition:x},~\eqref{Condition:har} \label{Res:11} \\
&\frac{1}{p}+\frac{1}{r}=\frac{2}{N} \;\;\qquad\quad {\rm from }\;~\eqref{Condition:14}\;
{\rm and}~\eqref{Condition:15},~\eqref{Condition:har} \label{Res:12}\\
&0<\frac{1}{p}-\frac{1}{r}<\frac{1}{N} \;\;\;\quad {\rm from}\;1/q<1/2\;\;{\rm in}\;~\eqref{Condition:14}\;{\rm and}\;~\eqref{Condition:15}  \label{Res:13}\\
&0<\frac{1}{p}-\frac{1}{r}< \frac{1}{2}(1+\frac{\gamma}{N})
\;\quad {\rm from}\;~\eqref{Condition:16}\;{\rm and}\;~\eqref{Condition:v} \label{Res:14}
\end{align}
\end{subequations}

Therefore
\[
\gamma=2-N\;,\;a=N.
\]
and  by~\eqref{pr-star-2},~\eqref{pr-star-1} and~\eqref{Condition:15}, 
\[
\frac{1}{N}<\frac{1}{p}<\frac{N+1}{N^2}\;,\;\frac{N-1}{N^2}<\frac{1}{r}<\frac{1}{N}.
\]
Thus we conclude that if the triplet $(q,r,p)$ satisfies~\eqref{solvable-g}, i.e.,  
\[
\{ (\mbq,\mbr,\mbp) | \;\frac{1}{\mbq}=\frac{N}{\mbp}-1\;,\;\frac{1}{\mbr}=\frac{2}{N}-\frac{1}{\mbp}\;,\;
\frac{1}{N}<\frac{1}{\mbp}<\frac{N+1}{N^2}\},
\]
then~\eqref{desire-est-sol} holds where the triplet $(\tmbq',\tmbr', \tmbp')$ is given by
~\eqref{Condition:v},~\eqref{Condition:x} ,~\eqref{Condition:har} and~\eqref{E:1t-variable}.   

\end{proof}

\begin{proof}[Proof of Corollary~\ref{pos-gain}]
When $f_{0}\geq 0$, we can see that the solution is non-negative by iterating Duhamel's formula: 
\begin{equation}\label{positive}
\begin{split}
f_+(t)&=U(t)f_{0}+\int_{0}^{t}U(t-t_{1})Q^{+}(U(t_{1})f_{0},U(t_{1})f_{0}) dt_{1} \\
&+\int_{0}^{t}\int_{0}^{t_{1}}U(t-t_{1})Q^{+}{\big (}U(t_{1}-t_{2})Q^{+}(U(t_{2})f_{0},U(t_{2})f_{0}
{\big )}, U(t_{1})f_{0})dt_{2}dt_{1}+\cdots
\end{split}
\end{equation}
Since each term in the right hand side is non-negative, it suffices to show the series converges. 
Using Strichartz estimates of Proposition~\ref{Strichartz-est} with solvable triplet repeatedly, 
we get that the Strichartz norm of $f_+$ for 
solvable triplets is bounded by a series of  $L^N_{x,v}$ norm of $f_0$ which converges since initial data is 
small enough. 
\end{proof}

Now we explain the reason why the above approach cannot solve the full Boltzmann equation. First we record the 
estimate for the loss term whose proof can be obtained by dropping $\ell$ and $m$ and modifying
 the proof of Proposition~\ref{P:Loss}.

\begin{lem}\label{T:Loss}
Let $1<\mt{p}, \mt{q}, \mt{r} <\infty$ with $-N<\gamma< 0$, $1/\mt{p}+1/\mt{q}=1+\gamma/N+1/\mt{r}$ and 
$1/\mt{p}<1/\mt{r}$. Assume the kernel~\eqref{D:kernel}{\rm:} 
\[
B(v-v_*,\omega)=|v-v_*|^{\gamma}b(\cos\theta)
\]
with $b(\cos\theta)$ satisfies  Grad's cutoff assumptation~\eqref{D:Grad}. Then the  bilinear operator $Q^{-}$ is a
bounded operator from $L^{\mt{p}}(\mathbb{R}^N)\times L^{\mt{q}}(\mathbb{R}^N)\rightarrow 
L^{\mt{r}}(\mathbb{R}^N)$ via the estimate 
\[
\|Q^{-}(f,g)\|_{L^{\mt{r}}(\mathbb{R}^N)}\leq C\|f\|_{L^{\mt{p}}(\mathbb{R}^N)}\|g\|_{L^{\mt{q}}(\mathbb{R}^N)}.
\]
\end{lem}

Please note the difference between the estimates for gain and loss terms, Proposition~\ref{P:Convolution} 
and Lemma~\ref{T:Loss}, the latter does not include the case $\gamma=0$ and it needs additional constraint
$1/\mt{p}<1/\mt{r}$. It is the main reason for equation~\eqref{desire-est-sol} does 
not hold for $Q^{-}$. In the proof of~\eqref{desire-est-sol}, 
we use~\eqref{G-convolution-est} by letting $\mt{r}=\tmbp'$ and $\mt{p}=\mt{q}=\mbp$.  The computation there also
shows that we need to take $\gamma=2-N$ and $1/N<1/\mbp<(N+1)/N^2$. However the range of $\mbp$ and equality 
$2/\mbp=1+(2-N)/N+1/\tmbp'$ exclude the possibility of $1/\mbp<1/\tmbp'$, i.e., $1/\mt{p}<1/\mt{r}$.

To end this subsection, we present two results which are useful in closing the Kaniel-Shinbrot  iteration.  
\begin{prop}\label{P-p2-est}
Let $N=3$ and $a_2=15/8$. Under the same assumption as Proposition~\ref{Thm-Gain}, if we further assume that
$\|f_0\|_{L^{a_2}_{x,v}}<\infty$, then solution $f_+$ in Proposition~\ref{Thm-Gain} also satisfies
\begin{equation}
\|f_+\|^2_{L^{q_2}_t L^{r_2}_x L^{p_2}_v}<\infty,
\end{equation} 
where $(1/q_2,1/r_2,1/p_2)=(1/2,\;11/30,\;21/30)$ is a KT-admissible triplet with $1/p_2+1/r_2=2/a_2$. 
\end{prop}
\begin{proof}  Due to the assumption that $\|f_0\|_{L^{a_2}_{x,v}}<\infty$,
we claim that there exist  KT-admissible triplets $(q_{2},r_{2},p_{2})$ and
$(\tilde{q}_{2},\tilde{r}_{2},\tilde{p}_{2})$ such that 
\begin{equation}\label{Gain-p2}
\|Q^+(f_+,f_+)\|_{L^{\tilde{q}'_{2}}_tL^{\tilde{r}'_{2}}_xL^{\tilde{p}'_{2}}_v} \leq C
\|f_+\|_{L^{\mbq}_tL^{\mbr}_xL^{\mbp}_v}\|f_+\|_{L^{q_{2}}_tL^{r_{2}}_xL^{p_{2}}_v},
\end{equation}
and $a_{2}=HM(p_{2},r_{2})=HM(\tilde{p}'_{2},\tilde{r}'_{2})$ where $\tilde{p}'_2$ means 
the conjugate of $\tilde{q}_2$ and so on.
From this together with Strichartz estimate \eqref{E:Strichartz}, we have
\begin{equation*}
\begin{split}
\|f_+\|_{L^{q_{2}}_tL^{r_{2}}_xL^{p_{2}}_v}&=\|Sf_+\|_{L^{q_{2}}_tL^{r_{2}}_xL^{p_{2}}_v}\\
& \leq C_{0}\|f_{0}\|_{L^{a_{2}}_{x,v}}+C_{1}\|Q^+(f_+,f_+)\|
_{L^{\tilde{q}'_{2}}_tL^{\tilde{r}'_{2}}_xL^{\tilde{p}'_{2}}_v} \\
&  \leq C_{0}\|f_{0}\|_{L^{a_{2}}_{x,v}}+C_{2}\|f_+\|_{L^{\mbq}_tL^{\mbr}_xL^{\mbp}_v}\|f_+\|_{L^{q_{2}}_tL^{r_{2}}_xL^{p_{2}}_v}.
\end{split}
\end{equation*}
The proof of Proposition~\ref{Thm-Gain} implies that $C_2\|f_+\|_{L^{\mbq}_tL^{\mbr}_xL^{\mbp}_v}<1$ 
 where $(\mbq,\mbr,\mbp)$ is a solvable triplet.  Thus we have 
\[
\|f_+\|_{L^{q_{2}}_tL^{r_{2}}_xL^{p_{2}}_v} \leq C_3 \|f_{0}\|_{L^{a_{2}}_{x,v}}<\infty.
\]
\indent To prove \eqref{Gain-p2}, 
we define $\tilde{p}'_{2}$ and $\tilde{r}'_{2}$ as  follows: 
\begin{equation}\label{choose-p-2}
\frac{1}{\tilde{p}'_{2}}:=\frac{1}{\mbp}+\frac{1}{30},\;\frac{1}{\tilde{r}'_{2}}:=\frac{1}{\mbr}+\frac{11}{30},\;
\frac{1}{\tilde{q}'_{2}}=\frac{1}{2}+\frac{1}{\mbq}.
\end{equation}
 By \eqref{G-convolution-est}, it is easy to have that 
\begin{subequations}
\begin{align}
&\frac{1}{\mbp}+\frac{1}{p_{2}}=\frac{2}{3}+\frac{1}{\tilde{p}'_{2}}\;,\; 
\frac{1}{\mbp}<\frac{1}{\tilde{p}'_{2}}<\frac{1}{p_{2}}\;, \;\label{uni-v}\\
&\frac{1}{\mbr}+\frac{1}{r_{2}}=\frac{1}{\tilde{r}'_{2}} \;, \label{uni-x}\\
&\frac{1}{p_{2}}+\frac{1}{r_{2}}=\frac{1}{\tilde{p}'_{2}}+\frac{1}{\tilde{r}'_{2}}=\frac{2}{a_{2}}=\frac{16}{15}\;, \label{uni-har}\\
&\frac{1}{q_{2}}=\frac{3}{2}(\frac{1}{p_{2}}-\frac{1}{r_{2}})
\;,\;\frac{1}{\tilde{q}_{2}}=\frac{3}{2}(\frac{1}{\tilde{p}_{2}}-\frac{1}{\tilde{r}_{2}})\;,\;\qquad\qquad\;\label{uni-q}. \\
&\frac{1}{\mbq}+\frac{1}{q_{2}}= \frac{1}{\tilde{q}'_{2}}.\;\label{uni-time}
\end{align}
\end{subequations}
and the proof of \eqref{Gain-p2} is finished.
\end{proof}
\begin{rem}
Our choice of $p_{2},\; r_{2},\;\tilde{p}_{2},\;\tilde{r}_{2}$ in~\eqref{choose-p-2} is one of possible combinations 
which satisfies~(\ref{uni-v},b,c,d,e). In fact, by~\eqref{pr-star-2}, i.e., $1/r_{2}<1/p_{2}<2/r_{2}$, 
and~\eqref{uni-v}, it is easy to find more options. 
\end{rem}

By \eqref{G-convolution-est}, the condition $1/\mbp<1/\tilde{p}'_{2}<1/p_{2}$ can be removed for \eqref{Gain-p2}
as an estimate for the gain term. However it is compulsory for the   loss term due to  Lemma~\ref{T:Loss}. In summary, we have the following result. 
\begin{cor}\label{C-L-p2} 
Use the same notations as Proposition~\ref{P-p2-est}. 
Suppose $f_{1}\in L^{\mbq}_tL^{\mbr}_xL^{\mbp}_v$ and $f_{2}\in L^{q_{2}}_tL^{r_{2}}_xL^{p_{2}}_v$, then 
$Q^{\pm}(f_{1},f_{2})\in L^{\tilde{q}'_{2}}_tL^{\tilde{r}'_{2}}_xL^{\tilde{p}'_{2}}_v$ and 
\[
\begin{split}
& \|Q^-(f_1,f_2)\|_{L^{\tilde{q}'_{2}}_tL^{\tilde{r}'_{2}}_xL^{\tilde{p}'_{2}}_v} \leq C
\|f_1\|_{L^{\mbq}_tL^{\mbr}_xL^{\mbp}_v}\|f_2\|_{L^{q_{2}}_tL^{r_{2}}_xL^{p_{2}}_v},\\
&\|Q^+(f_1,f_2)\|_{L^{\tilde{q}'_{2}}_tL^{\tilde{r}'_{2}}_xL^{\tilde{p}'_{2}}_v} \leq C
\|f_1\|_{L^{\mbq}_tL^{\mbr}_xL^{\mbp}_v}\|f_2\|_{L^{q_{2}}_tL^{r_{2}}_xL^{p_{2}}_v},\\
&\|Q^+(f_1,f_2)\|_{L^{\tilde{q}'_{2}}_tL^{\tilde{r}'_{2}}_xL^{\tilde{p}'_{2}}_v} \leq C
\|f_2\|_{L^{\mbq}_tL^{\mbr}_xL^{\mbp}_v}\|f_1\|_{L^{q_{2}}_tL^{r_{2}}_xL^{p_{2}}_v}.
\end{split}
\]
\end{cor}

\section{Back to the full equation}\par

\subsection{Well-defined of Loss term and Kaniel and Shinbrot's iteration}\label{sub-loss}

We will follow the idea of Chen, Denlinger and Pavlovi\'{c}~\cite{CDP21} to recover 
the solutions to the full Boltzmann equation from the solutions to the gain term only Boltzmann equation
by making use of Kaniel and Shinbrot's iteration ~\cite{KS78,IS84}.

\begin{prop}\label{P-solve-full-1} 
Consider the Cauchy problem~\eqref{E:Cauchy} with $N=3,\;\gamma=-1$. Suppose the initial data
$$
f_{0}\in {\mathbb B}_{\eta}=\{f_{0} | f_0\geq 0,\; \|f_{0}\|_{L^{3}_{x,v}}<\eta,\;
\|f_0\|_{L^{15/8}_{x,v}}<\infty\}\subset L^3_{x,v},
$$
where $\eta$ is chosen in Propositions~\ref{Thm-Gain}.  Then ~\eqref{E:Cauchy} admits a non-negative unique 
mild solution 
 \[
f\in C([0,\infty),L^3_{x,v})\cap
L^{\mbq}([0,\infty],L^{\mbr}_xL^{\mbp}_v) 
\]
where the triple $(\mbq,\mbr,\mbp)$ lies in the set~\eqref{solvable}.
The solution map $f_0\in B_{\eta}\subset L^3_{x,v}\rightarrow f\in L^{\mbq}_tL^{\mbr}_xL^{\mbp}_v$ 
is Lipschitz continuous and the solution $f$ scatters with respect to the kinetic 
transport operator in $L^3_{x,v}$. 
\end{prop}
\begin{proof} Let us denote the loss term $Q^{-}(f_{1},f_{2})=f_{1}L(f_{2})$.
First we recall that the Kaniel-Shinbrot iteration ensures that if there exist measurable functions $h_1,h_2,g_1,g_2$ which 
satisfy the beginning condition, i.e.,   
\begin{equation}\label{begin-condition}
 0\leq h_{1}\leq h_{2}\leq g_{2}\leq g_{1},
\end{equation}
then the iteration($n\geq 2$)
\begin{equation}\label{Iteration}
\begin{split}
& \partial_{t} g_{n+1}+v\cdot\nabla_{x} g_{n+1}+ g_{n+1}L(h_{n})=Q^{+}(g_{n},g_{n}) \\
& \partial_{t} h_{n+1}+v\cdot\nabla_{x} h_{n+1}+ h_{n+1} L (g_{n})=Q^{+}(h_{n},h_{n}) \\
& g_{n+1}(0)=h_{n+1}(0)=f_{0}
\end{split}
\end{equation}
will induce the monotone sequence of measurable functions
\begin{equation}\label{monotone-sequence}
0\leq h_1\leq h_n\leq h_{n+1}\leq g_{n+1}\leq g_n\leq g_1.
\end{equation}
Thus the monotone convergence theorem implies the existence of the limits $g, h$ with $0\leq h\leq g\leq g_1$
which satisfy 
\begin{equation}\label{Iteration-limit}
\begin{split}
& \partial_{t} g+v\cdot\nabla_{x} g+ g L(h )=Q^{+}(g,g) \\
& \partial_{t} h+v\cdot\nabla_{x} h+ h L (g)=Q^{+}(h,h) \\
& g(0)=h(0)=f_{0}
\end{split}
\end{equation}
Hence the Cauchy problem~\eqref{E:Cauchy} is solved if one can further prove $g=h(=f)$.

Based on the Proposition~\ref{Thm-Gain}, it is natural to choose $h_1\equiv 0$ and $g_1=f_{+}$ where $f_{+}\geq 0$
is the solution of gain term only Boltzmann equation~\eqref{G-Cauchy} with initial data  $\|f_{0}\|_{L^{3}_{x,v}}<\eta$. 
The Proposition~\ref{Thm-Gain}  ensures
\begin{equation}\label{f-plus-finite}
f_+\in C([0,\infty),L^3_{x,v})\cap
L^{\mbq}([0,\infty],L^{\mbr}_xL^{\mbp}_v) 
\end{equation}
where $(\mbq,\mbr,\mbp)$ is a solvable triplet (i.e., satisfying~\eqref{solvable-g}).  

According to~\eqref{Iteration}, we want to 
find $h_{2}$ and $g_{2}$ through 
\begin{equation}\label{h-g-2}
\begin{split}
& \partial_{t} g_{2}+v\cdot\nabla_{x} g_{2}+ g_{2}L(h_{1})=Q^{+}(g_{1},g_{1}) \\
& \partial_{t} h_{2}+v\cdot\nabla_{x} h_{2}+ h_{2} L (g_{1})=Q^{+}(h_{1},h_{1}) \\
& g_{2}(0)=h_{2}(0)=f_{0}
\end{split}
\end{equation}
Since $h_{1}\equiv 0$, the first equation of~\eqref{h-g-2}  is exactly the gain term only Boltzmann equation. Hence 
we have $g_{2}=g_{1}=f_{+}$ by Proposition~\ref{Thm-Gain} and
\begin{equation}\label{g-2}
g_{2}(t)=U(t)f_{0}+\int_{0}^{t} U(t-s)Q^{+}(g_{1},g_{1})(s)ds. 
\end{equation}

Next we want to solve the second equation of~\eqref{h-g-2}  with the given $g_{1}$. More precisely, formally we have 
\begin{equation}\label{h-2}
  h_{2}(t)=U(t)f_{0}\; e^{-\int_{0}^{t} U(t-s) L(g_{1})(s) ds}.
\end{equation}
To ensure that $L(g_{1})$ is pointwisely a.e. well-defined  when  $g_{1}=f_{+}$ 
satisfies~\eqref{f-plus-finite}, we recall that the assumption $\|f_{0}\|_{L^{15/8}_{x,v}}<\infty$ and Proposition~\ref{P-p2-est}
ensure 
\begin{equation}\label{p2boundf+}
g_{1}=f_{+}\in L^{q_{2}}([0,\infty],L^{r_{2}}_xL^{p_{2}}_v).
\end{equation}
Note that we are looking for a solution $h_{2}\in L^{\mbq}([0,\infty],L^{\mbr}_xL^{\mbp}_v)$. From the estimate 
of  Corollary~\ref{C-L-p2}, 
\[
\|h_{2}L(g_1)\|_{L^{\tilde{q}'_{2}}_tL^{\tilde{r}'_{2}}_xL^{\tilde{p}'_{2}}_v} \leq C
\|h_{2}\|_{L^{\mbq}_tL^{\mbr}_xL^{\mbp}_v}\|g_1\|_{L^{q_{2}}_tL^{r_{2}}_xL^{p_{2}}_v},
\] 
we know that if $\phi\in L^{\tilde{q}_{2}}_{t}L^{\tilde{r}_{2}}_{x}L^{\tilde{p}_{2}}_{v}$ then $\lr{h_{2}L(g_{1}),\phi}$
is bounded. Since $\lr{h_{2}L(g_{1}),\phi}=\lr{L(g_{1}),h_{2}\phi}$, we have $L(g_{1})\in L^{\mbq_{2}}_{t}L^{\mbr_{2}}_{x}L^{\mbp_{2}}_{v}$ where $(1/\mbq_{2})'=1/\tilde{q}_{2}+1/\mbq$
, $(1/\mbr_{2})'=1/\tilde{r}_{2}+1/\mbr$ and $(1/\mbp_{2})'=1/\tilde{p}_{2}+1/\mbp$. Therefore $L(g_{1})$ is pointwisely 
a.e. well-defined.

Now we can compute $h_2$ by~\eqref{h-2}. It is easy to  have $h_1\equiv 0\leq h_2\leq U(t)f_0\leq g_2$ by non-negativity of $g_1$.  Therefore we conclude
the beginning condition~\eqref{begin-condition}.  From~\eqref{p2boundf+}, we also have that
\[
h_{2},\;g_{2}\in L^{q_{2}}([0,\infty],L^{r_{2}}_xL^{p_{2}}_v).
\]
Thus we can repeat the above argument to check that each term in~\eqref{Iteration} is well-defined. 
Therefore the method of Kaniel-Shinbrot 
ensures the existence of monotone sequence~\eqref{monotone-sequence} and the limit functions $g,\;h$ satisfy~\eqref{Iteration-limit}.

We also note that from the monotone convergence theorem and~\eqref{monotone-sequence}, we have
\begin{equation}\label{g-h-finite}
\begin{split}
& g,\;h \in  L^{\mbq}([0,\infty],L^{\mbr}_xL^{\mbp}_v), \\
& g,\;h \in  L^{q_{2}}([0,\infty],L^{r_{2}}_xL^{p_{2}}_v), \\
& Q^{+}(g,g), \;Q^{+}(h,h)\in L^{\tmbq'}_tL^{\tmbr'}_xL^{\tmbp'}_{v},\\
& Q^{\pm}(g,g),\;Q^{\pm}(h,h)\in L^{\tilde{q}'_{2}}_tL^{\tilde{r}'_{2}}_xL^{\tilde{p}'_{2}}_{v}.
\end{split}
\end{equation}
Thus we have a solution for the full Boltzmann equation if $g=h$.  

To prove that $g=h$, we let $w=g-h\geq 0$.  By~\eqref{Iteration-limit} the difference $w$ satisfies the equation
\[
 \partial_t w+v\cdot\nabla_x w=Q^+(g,w)+Q^+(w,h)+Q^-(g,w)-Q^-(w,g)
\]
with zero initial data. By   Lemma~\ref{w-equation} below we know that this equation has a unique solution 
$w\equiv 0$.  Thus $g=h$ and we conclude the global existence of the  non-negative mild solution for the 
full Boltzmann equation. The uniqueness of this solution can be proved by a standard continuity argument 
and the fact that the solutions are non-negative. We include it in subsection~\ref{s-unique-sol}.  Also the continuity 
in time, scattering of the solution and Lipschitz continuous of the solution map is included in the subsection
~\ref{s-continuity}.
 Thus we conclude the Proposition~\ref{P-solve-full-1}.
\end{proof}

\begin{lem}\label{w-equation}
Let $g,\;h$ be non-negative functions satisfy~\eqref{g-h-finite}.  Suppose that $w\geq 0$ is a mild solution 
of 
\begin{equation}\label{g-h-equation}
\left\{
\begin{array}{l}
 \partial_t w+v\cdot\nabla_x w=Q^+(g,w)+Q^+(w,h)+Q^-(g,w)-Q^-(w,g) \\
 w(0)=0.
\end{array}
\right.
\end{equation}
Then $w\equiv 0$. 
\end{lem}
\begin{proof}

Consider the given time interval $[0,T]$ and define
\[
t_{0}=\inf {\Big \{} t\in [0,T]{\Big |} \|w(t)\|_{L^{q_{2}}([0, t],L^{r_{2}}_{x}L^{p_{2}}_{v})}>0 {\Big \}}.
\]
 Then $w\equiv 0$ for $0\leq t\leq t_{0}$. Let $t_{0}\leq s\leq T$.
%We define, for $0\leq s\leq T-t_{0}$, 
%\[
%d_{t_{0}}(s)=\| w\|_{L^{q_{2}}([t_{0}, s],L^{r_{2}}_{x}L^{p_{2}}_{v})}.
%\]

From~\eqref{g-h-equation}, we have 
\[
 w(s)=\int_{0}^{s} U(t-\tau){\big [}Q^{+}(g,w)+Q^{+}(w,h)+Q^{-}(g,w)-Q^{-}(w,g){\big ]}(\tau) d\tau. 
\]
Noting that  $w\ge 0$,  $0\leq h\leq g$ and the operators $U(t), Q^{+}, Q^{-}$ are non-negative,  we have 
\begin{equation}\label{w-zero}
0\leq  w(s)\leq \int_{0}^{s} U(t-\tau){\big [}Q^{+}(g,w)+Q^{+}(w,h)+Q^{-}(g,w){\big ]}(\tau) d\tau. 
\end{equation}
Apply Strichartz estimates as Proposition~\ref{P-p2-est} and use the estimate of Corollary~\ref{C-L-p2}, then we have  
\begin{equation}\label{w-est}
\begin{split}
&\|w\|_{L^{q_{2}}([t_{0},s],L^{r_{2}}_xL^{p_{2}}_v)}\\
&\leq C( \| g\|_{L^{\mbq}([t_{0},s],L^{\mbr}_xL^{\mbp}_v)}+ \| h\|_{L^{\mbq}([t_{0},s],L^{\mbr}_xL^{\mbp}_v)})
\| w\|_{L^{q_{2}}(([t_{0},s],L^{r_{2}}_xL^{p_{2}}_v)}\\
&:=C(g,h,s) \| w\|_{L^{q_{2}}([t_{0},s],L^{r_{2}}_xL^{p_{2}}_v)}.
\end{split}
\end{equation}
Letting $s \rightarrow t_{0}+$, clearly we have $C(g,h,s)<1$ which means that 
\[
\| w\|_{L^{q_{2}}([t_{0},s],L^{r_{2}}_xL^{p_{2}}_v)}=0
\] for 
some $s>0$. By continuity, we have $w|_{[0,T]}\equiv 0$ for any $T>0$.
\end{proof}

\subsection{Uniqueness of the solution}\label{s-unique-sol}
Assume that $g\geq 0$ and $ h\geq 0$ both are mild solutions which satisfy~\eqref{E:Cauchy}. Let 
$w=g-h$. Comparing to Lemma~\ref{w-equation}, the function $w$ satisfies~\eqref{g-h-equation}, 
but we do not have the property $ w\geq 0$.  For our convenience, we rewrite~\eqref{g-h-equation} as
\[
\left\{
\begin{array}{l}
 \partial_t w+v\cdot\nabla_x w+wL(g)=Q^+(g,w)+Q^+(w,h)+Q^-(g,w) \\
 w(0)=0,
\end{array}
\right.
\]
and want to show $w\equiv 0$. Since $g,\;h$ and thus $w$ satisfy~\eqref{g-h-finite}, the term $L(g)$ is pointwisely
a.e. well-defined. Thus  the function $w$ satisfies 
\[
\begin{split}
& w(t)=\int_0^t  e^{-\int_s^t U(t-\tau)L(g)(\tau)d\tau} \\
&\hskip 3cm \cdot U(t-s){\big [} Q^+(g,w)+Q^+(w,h)+Q^-(g,w) {\big]}(s)ds.
\end{split}
\]
Using the  fact that $L(g)\geq 0$ since $g\geq 0$, we have 
\begin{equation}\label{w-uni}
| w(t) | \leq   \int_0^t   U(t-s){\big [} Q^+(g,|w|)+Q^+(|w|,h)+Q^-(g,|w|) {\big]}(s)ds.
\end{equation}
The equation~\eqref{w-uni} is in place of~\eqref{w-zero}  for the proof of $w\equiv 0$, thus we conclude 
the uniqueness of the solution.

\subsection{Continuity in $L^{3}_{x,v}$ and Scattering of the solution}\label{s-continuity}

Now we show that $f\in C([0,T],L^3_{x,v})$ for any $T\in[0,\infty]$. From the formula 
\begin{equation}\label{Duhamel-positive}
\int_{0}^{t} U(t-s)Q^{-}(f,f)(s)ds+f(t)=U(t)f_{0}+\int_{0}^{t} U(t-s)Q^{+}(f,f)(s)ds
\end{equation}
and the observation that each term in~\eqref{Duhamel-positive} is non-negative, we have 
\begin{equation}\label{Duhamel-ineq}
0\leq f(t)\leq U(t)f_{0}+\int_{0}^{t} U(t-s)Q^{+}(f,f)(s)ds.
\end{equation} 
It has been observed by Ovcharov~\cite{Ovc09} that  
$U(t)f_0\in C(\mathbb{R};L^N_{x,v})$, hence it suffices to show that $W(t)$ (see \eqref{D:solution-maps}) is also continuous. 
Let $0\leq t\in (0,\infty]$.  
Applying inhomogeneous Strichartz with triplet
$(\tmbq',\tmbr',\tmbp')$ used in~\eqref{desire-est-sol}, we see that 
\[
\|W(t)Q^{+}(f,f)\|_{L^{\infty}([0,t];L^3_{x,v})}=\int_0^t \|U(t-s)Q^{+}(f,f)\|_{L^3_{x,v}}ds
\]
is bounded.  Since $U(t)$ is continuous, we conclude that $W(t)$ is continuous from above expression. 
Also the solution map $f_0\in B_{\eta}\subset L^3_{x,v}\rightarrow f\in L^q_tL^r_xL^p_v$ is Lipschitz continuous.

Next we want to show that the solution $f$ scatters, i.e., there exists  a function $f_{\infty}\in L^{3}_{x,v}$ such that 
\[
\|f(t)-U(t)f_{\infty}\|_{L^3_{x,v}}\rightarrow 0\;\;{\rm as}\;t\rightarrow \infty.
\]
The above statement is equivalent to prove that 
\begin{equation}
\|U(-t)f(t)-f_{\infty}\|_{L^3_{x,v}}\rightarrow 0\;\;{\rm as}\;t\rightarrow \infty,
\end{equation}
since $U(t)$ preserves the $L^3_{x,v}$ norm. 

By the Duhamel formula, we have 
\begin{equation}\label{Duhamel}
U(-t)f(t)=f_0+\int_0^t U(-s)Q(f,f)(s)ds.
\end{equation}
Hence the scattering of $f(t)$ is confirmed if we have  the convergence of the integral 
\[
\int_0^{\infty} U(-t)Q(f,f)(t)dt
\]  
in $L^{3}_{x,v}$. In this case $f_{\infty}$ is given by
\begin{equation}\label{f-infty}
f_{\infty}=f_0+\int_0^{\infty} U(-t)Q(f,f)(t)dt.
\end{equation}

We rewrite~\eqref{Duhamel} as
\[
U(-t)f(t)+\int_0^t U(-s)Q^{-}(f,f)(s)ds=f_0+\int_0^t U(-s)Q^{+}(f,f)(s)ds.
\]
 Since each term of above equation is non-negative, thus we have 
 \[
 \begin{split}
 \int_0^t U(-s)Q^{-}(f,f)(s)ds& \leq f_0+\int_0^t U(-s)Q^{+}(f,f)(s)ds\\
 &\leq f_0+\int_0^{\infty} U(-s)Q^{+}(f,f)(s)ds.
\end{split}
\]
By monotone convergence theorem, it holds that
\begin{equation}\label{Duhamel-p}
 \int_0^{\infty} U(-s)Q^{-}(f,f)(s)ds\leq f_0+\int_0^{\infty} U(-s)Q^{+}(f,f)(s)ds.
\end{equation} 
Then we are reduced to prove the right hand side of~\eqref{Duhamel-p} is bounded 
in $L^{3}_{x,v}$.
 
 Let $U^*(t)$ be the adjoint operator of $U(t)$, it is clearly that $U^*(t)=U(-t)$. 
Let $(\tmbq,\tmbr,\tmbp)$ be the KT-admissible triplet chosen in the proof of~\eqref{desire-est-sol}
and recall that $1/\tmbp'+1/\tmbr'=2/3$. 
By duality, the homogeneous Strichartz estimate 
\[
\|U(t)g\|_{L^{\tmbq}_tL^{\tmbr}_xL^{\tmbp}_v}\leq C\|g\|_{L^{{3}/{2}}_{x,v}}
\]
implies
\[
\bigg\|\int_0^{\infty} U^*(t)Q^+(f,f)dt \bigg\|_{L^{3}_{x,v}}\leq C\|Q^+(f,f)\|_{L^{\tmbq'}_tL^{\tmbr'}_xL^{\tmbp'}_v}
\leq C\|f\|^{2}_{L^{\mbq}_tL^{\mbr}_xL^{\mbp}_v} 
\]
as before. Thus we conclude that $f$ scatters.

\subsection{Proof of Theorem~\ref{T3} }
\begin{proof}
As before, we apply the Strichartz estimate to the gain term only Boltzmann equation 
\[
\|Sf_+(t,x,v)\|_{L^q([0,T];L^r_xL^p_v)}\leq C{\big (}\|f_0\|_{L^a_{x,v}}+
\|Q(f_+,f_+)\|_{L^{\tilde{q}'}([0,T];L^{\tilde{r}'}_xL^{\tilde{p}'}_v)}{\big )}.
\]
To show that
\begin{equation}\label{local-contraction}
\begin{split}
& \|Sf_+(t,x,v)\|_{L^q([0,T];L^r_xL^p_v)}\\
& {\hskip 1cm}\leq C_1\|f_0(x,v)\|_{L^a_{x,v}}+ C_2 T^{\beta} 
\|f(t,x,v)\|^2_{L^q([0,T];L^r_xL^p_v)}
\end{split}
\end{equation}
with $\beta>0$,  we need to find KT-admissible triplets $(q,r,p)$ and $(\tilde{q},\tilde{r},\tilde{p})$
which satisfy
\[
\frac{2}{p}=1+\frac{\gamma}{N}+\frac{1}{\tilde{p}'}\;,\; \frac{2}{r}=\frac{1}{\tilde{r}'}\;,\;\frac{1}{p}+\frac{1}{r}
=\frac{1}{\tilde{p}'}+\frac{1}{\tilde{r}'}
\]
and
\[
\frac{2}{q}+\beta=\frac{1}{\tilde{q}'}.
\] 
It is already found in~\cite{HJ17}  that the set~\eqref{solvable-2} is the collection of all possible 
$(q,r,p)$ where
\[
\beta=\frac{(2-N)-\gamma}{2}>0.
\] 
To recover the solution for the full equation, we   need another pair of KT-admissible triplets $(q_2,r_2,p_2)$ and $(\tilde{q}'_2,\tilde{r}'_2,\tilde{p}'_2)$
. It is straightforward to check that the following choice works:
\[
\begin{split}
& \frac{1}{p_{2}}=2\alpha {\Big (}\frac{\gamma+N}{N} {\Big )}\;,\;
\frac{1}{r_{2}}=(3\alpha-1){\Big (}\frac{\gamma+N}{N} {\Big )} \\
&\frac{1}{\tilde{p}'_{2}}=(3\alpha-1){\Big (} \frac{\gamma+N}{N}{\Big )}\;,\;
\frac{1}{\tilde{r}'_{2}}=2\alpha {\Big (}\frac{\gamma+N}{N} {\Big )}.
\end{split}
\]
\end{proof}

\section{Theorem~\ref{result2}: The case $-1<\gamma<0$}

In this section, we give the proof for Theorem~\ref{result2}. It is in the same spirit as that of the case $\gamma=-1$ 
except that we need the weighted estimates for the gain and loss terms as well as the weighted Strichartz estimates.

\subsection{Weighted estimates} Let $\lr{v}=(1+|v|^{2})^{1/2}$.
To prove the weighted estimates for the gain term, we  consider the quantity 
\begin{equation}\label{gain-dual}
\lr{\lr{v}^{\ell}Q^{+}(f,g),\lr{v}^{-\ell}\psi},\;\ell>0.
\end{equation}
When $\ell=0$, Alonso, Carneiro and Gamba~\cite{ACG10} introduce a bilinear operator to give~\eqref{gain-dual} 
two representations which are used to prove the estimates  collected in Proposition~\ref{P:Convolution}
(see the upcoming proof of Proposition~\ref{P:Convolution-w} for two representations). In what follows,
we first prove that     the quantity~\eqref{gain-dual} with $\ell\neq0$  can also be bounded by the formulas with the same representations as $\ell=0$. Then the desired  estimate  follows.

\begin{prop}\label{P:Convolution-w} 
Let $\ell\geq 0$, $1<\mt{p}, \mt{q}, \mt{r} <\infty$ with $-N<\gamma\leq 0$ and
\begin{equation}\label{scaling-relation}
1/\mt{p}+1/\mt{q}=1+\gamma/N+1/\mt{r}. 
\end{equation}
Assume the kernel~\eqref{D:kernel}{\rm:} 
\[
B(v-v_*,\omega)=|v-v_*|^{\gamma}b(\cos\theta)
\]
with $b(\cos\theta)$ satisfies Grad's cutoff assumption~\eqref{D:Grad}.
Then the bilinear operator $Q^{+}(f,g)$ satisfies  
\begin{equation}\label{W-convolution}
\|\lr{v}^{\ell}Q^{+}(f,g)\|_{L^{\mt{r}}_{v}(\mathbb{R}^N)}\leq C\|\lr{v}^{\ell} f\|_{L^{\mt{p}}_{v}(\mathbb{R}^N)}
\|\lr{v}^{\ell}g\|_{L^{\mt{q}}_{v}(\mathbb{R}^N)}.
\end{equation}
If $\ell>N/m$ and $1<\mt{p}_{m}, \mt{q}_{m}, m, \mt{r}_{m} <\infty$ satisfy
\begin{equation}\label{minus-size-condition}
\frac{1}{\mt{p}_{m}}+\frac{1}{m}<1\;,\;\frac{1}{\mt{q}_{m}}+\frac{1}{m}<1,
\end{equation}
and
\begin{equation}\label{scaling-relation-w}
\frac{1}{\mt{p}_{m}}+\frac{1}{\mt{q}_{m}}+\frac{1}{m}=1+\frac{\gamma}{N}+\frac{1}{\mt{r}_{m}}, 
\end{equation}
then we have  
\begin{equation}\label{W-convolution-2}
\|\lr{v}^{\ell} Q^{+}(f,g)\|_{L^{\mt{r}_{m}}}\leq C(\mt{p}_{m},\ell) \|\lr{v}^{\ell} f\|_{L^{\mt{p}_{m}} }\|\lr{v}^{\ell} g\|_{L^{\mt{q}_{m}}}.
\end{equation}
\end{prop}
\begin{rem}
Note that we use different notations  $\mt{p},\mt{q},\mt{r}$ and $\mt{p}_{m}, \mt{q}_{m}, \mt{r}_{m}$ to differ~\eqref{W-convolution}
from~\eqref{W-convolution-2} for their exponents satisfying different relations. 
\end{rem}
\begin{proof}[Proof of Proposition~\ref{P:Convolution-w}] 
First of all, we need to adopt the notations used in~\cite{ACG10}. Let
\[
\hat{u}=u/|u|,\; u=v-v_*.
\]
It is well known that the pre-post collision velocity relation~\eqref{E:pre-collision} is equivalent to  
\begin{equation}\label{sigma-rep}
v'=v-\frac{1}{2}(u-|u|\sigma),\;v'_*=v_*+\frac{1}{2}(u-|u|\sigma),
\end{equation}
and
\[
d\omega=\frac{1}{4\cos\theta} d\sigma
\]
where $\theta$ is the angle between $\omega$ and $u=v-v_*$.
In~\cite{ACG10}, the collision kernel is denoted by
\[
B(|u|,\hat{u}\cdot\sigma)=|u|^{\gamma}b(\hat{u}\cdot\sigma). 
\]
(Note that $\gamma$ above is denoted by $\lambda$ in~\cite{ACG10}.) They also define the bilinear operator
\begin{equation}\label{P}
\mathcal{P}(\psi,\phi)(u):=\int_{S^{N-1}}\psi(u^{-})\phi(u^{+})b(\hat{u}\cdot\sigma)\,d\sigma\,,
\end{equation}
where the variables $u^{+}$ and $u^{-}$ are defined by
\[
u^{-}:=\tfrac{1}{2}(u-|u|\sigma)\ \ \mbox{and}\ \ u^{+}:=u-u^{-}=\tfrac{1}{2}(u+|u|\sigma). 
\]
We note that the vector $\omega$ is often used in the occurrence~\eqref{E:pre-collision} and $\sigma$ 
in~\eqref{sigma-rep}, while  the vector $\sigma$ in~\eqref{sigma-rep} is denoted by $\omega$ in~\cite{ACG10}. 
However this difference of notations clearly does not affect the proof of any related result.
They also use  $\tau$ and $\mathcal{R}$ to denote the translation and reflection operators 
\[
\tau_{v}\psi(x):=\psi(x-v),\;\mathcal{R}\psi(x):=\psi(-x).
\]

Use above notations,  the representations 
\begin{equation}\label{dual-rep-P}
\begin{split}
\int_{\R^{N}} Q^{+}(f,g)(v)\; \psi (v) dv 
&=\int_{\R^{N}}\int_{\R^{N}} f(v)g(v-u) \mathcal{P}(\tau_{v}\mathcal{R}\psi,1)(u)|u|^{\gamma} dudv\\
&=\int_{\R^{N}}\int_{\R^{N}} f(u+v)g(v) \mathcal{P}(1, \tau_{-v}\mathcal{R}\psi)(u)|u|^{\gamma} dudv
\end{split}
\end{equation}
are used to prove the estimates collected in Proposition~\ref{P:Convolution}. More precisely, 
when $\gamma=0$, Alonso etc. used the first line of~\eqref{dual-rep-P} (equation (4.1) in~\cite{ACG10}) as a starting 
point to show~\eqref{W-convolution} with $\ell=0$ . When $-N<\gamma<0$, both lines of~\eqref{dual-rep-P} 
(equations (5.1) and (5.12) in~\cite{ACG10}) are  used   to show the 
case $\ell=0$ of~\eqref{W-convolution}.

On the other hand,  we follow the observation of Lions~\cite{Lio94} to write
\begin{equation}\label{dual-rep-T}
\begin{split}
\lr{Q^+(f,g),h}
&=\int_{\R^{N}}\int_{\R^{N}}\int_{S^{N-1}}  f(v)g(v_*) \psi(v') B(v-v_*,\omega) d\omega dv_* dv \\
&=\int_{\R^{N}}\int_{\R^{N}} f(v)g(v_*) (\tau_{-v_*}\circ T\circ \tau_{v_*}) \;\psi (v)dv_*dv.\\
\end{split}
\end{equation}
Here $T$ is a Radon transform 
\[
T \psi(x)=|x|^{\gamma}\int_{\omega\in S^{N-1}_+} b(\cos\theta)\;\psi(x-(x\cdot\omega)\omega) d\omega, 
\]
with
$\cos\theta={(x\cdot\omega)}/{|x|}, x\neq 0$, $x=|x|(0,0,1)$ and 
$\omega=(\cos\varphi\sin\theta,\sin\varphi\sin\theta,\cos\theta)$, $0\leq\theta\leq \pi/2 $. 
The regularizing effect of $T$ is first studied by Lions~\cite{Lio94}, then studied by several authors, 
see~\cite{JC12,JC20} for more details. 

Combining~\eqref{dual-rep-P} and~\eqref{dual-rep-T}, we have 
\begin{equation}\label{gain-dual-rep}
\begin{split}
\int Q^{+}(f,g)(v)\; \psi (v) dv
&=\iint f(v)g(v-u) \mathcal{P}(\tau_{v}\mathcal{R}\psi,1)(u)|u|^{\gamma} dudv\\
&=\iint  f(u+v)g(v) \mathcal{P}(1, \tau_{-v}\mathcal{R}\psi)(u)|u|^{\gamma} dudv\\
&=\iint f(v)g(v_*) (\tau_{-v_*}\circ T\circ \tau_{v_*})\psi(v)dv_*dv.
\end{split}
\end{equation}

With these preparations,  we are ready to estimate the quantity $\lr{Q^{+}(f,g),\psi}$. 
From the conservation of  energy, either $\lr{v'}\leq 2\lr{v}$ or $\lr{v'}\leq 2\lr{v_{*}}$ has to be true. Hence
for any $\ell\geq 0$ we have either
\begin{equation}\label{weight-ineq}
\lr{v}^{-\ell}\leq \lr{v'}^{-\ell} \;{\rm or}\;\lr{v_{*}}^{-\ell}\leq \lr{v'}^{-\ell}.
\end{equation}
We define $\Psi(v,v_*):=(\tau_{-v_*}\circ T\circ \tau_{v_*})\psi(v)$ and $\psi_{-\ell}(v)=\lr{v}^{-\ell}\psi(v)$. 
From~\eqref{weight-ineq}, we know that one of the followings estimates is true:
\begin{equation}\label{psi-weight-1}
\begin{split}
&{\big |} \Psi(v,v_*){\big |}
\leq \int_{S^{N-1}} {\big |} \psi (v'){\big |} B(v-v_*,\omega)d\omega \\
&\leq  \lr{v}^{\ell}\int_{S^{N-1}}  {\big |}\lr{v'}^{-\ell} \psi(v') {\big |} B(v-v_*,\omega)d\omega\\
&= \lr{v}^{\ell} (\tau_{-v_*}\circ T\circ \tau_{v_*}) {\big |} \psi_{-\ell} {\big |}(v),
\end{split}
\end{equation}
or  
\begin{equation}\label{psi-weight-2}
\begin{split}
&{\big |} \Psi(v,v_*){\big |}
\leq \int_{S^{N-1}} {\big |} \psi (v'){\big |} B(v-v_*,\omega)d\omega \\
&\leq  \lr{v_{*}}^{\ell} \int_{S^{N-1}}  {\big |}\lr{v'}^{-\ell} \psi(v') {\big |} B(v-v_*,\omega)d\omega\\
&= \lr{v_{*}}^{\ell} (\tau_{-v_*}\circ T\circ \tau_{v_*}) {\big |} \psi_{-\ell} {\big |}(v).
\end{split} 
\end{equation}

Denote $f_{\ell}(v)=\lr{v}^{\ell}f(v)$ and $g_{\ell}(v_{*})=\lr{v_{*}}^{\ell}g(v_{*})$. 
 Combining~\eqref{gain-dual-rep},~\eqref{psi-weight-1} and~\eqref{psi-weight-2} , we have
\begin{equation}\label{dual-weight}
\begin{split}
& {\Big |} \int  Q^{+}(f,g)(v)\; \psi (v) dv {\Big |} \\
&\leq \iint {\Big \{} {\big |} f_{\ell}(v)g(v_*){\big |}+{\big |}f(v)g_{\ell}(v_*) {\big |} {\Big \}} 
(\tau_{-v_*}\circ T\circ \tau_{v_*}){\big |}\psi_{-\ell}{\big |}(v)dv_*dv.
\end{split}
\end{equation}
Using the formulas of~\eqref{gain-dual-rep} to the right hand side of~\eqref{dual-weight}, we have  
\begin{equation}\label{P-ineq-1}
\begin{split}
&{\Big |}\int  Q^{+}(f,g)(v)\; \psi (v) dv {\Big |}\\
&\leq \iint {\Big \{} {\big |} f_{\ell}(v)g(v-u) {\big |}+ 
{\big |} f(v)g_{\ell}(v-u) {\big |}{\Big \}} \mathcal{P}(\tau_{v}\mathcal{R}|\psi_{-\ell}|,1)(u)|u|^{\gamma} dudv
\end{split}
\end{equation}
and 
\begin{equation}\label{P-ineq-2}
\begin{split}
&{\Big |}\int Q^{+}(f,g)(v)\; \psi (v) dv {\Big |}\\
&\leq \iint {\Big \{} {\big |} f_{\ell}(u+v)g(v) {\big |}+ 
{\big |} f(u+v)g_{\ell}(v) {\big |}{\Big \}} \mathcal{P}(1, \tau_{-v}\mathcal{R}|\psi_{-\ell}|)(u)|u|^{\gamma} dudv
\end{split}
\end{equation}
We note that the right hand sides of~\eqref{P-ineq-1} and~\eqref{P-ineq-2} still preserve the form of the 
representations in~\eqref{dual-rep-P}.  Following the proofs of~\cite{ACG10}, we  have 
\begin{equation}\label{dual-w-est}
\begin{split}
&{\Big |}\int  Q^{+}(f,g)(v)\; \psi (v) dv {\Big |}\\
&\leq C{\big (} \|f_{\ell}\|_{L^{\mt{p}}_{v}(\R^{N})}\|g\|_{L^{\mt{q}}_{v}(\R^{N})} + 
\|f\|_{L^{\mt{p}}_{v}(\R^{N})}\|g_{\ell}\|_{L^{\mt{q}}_{v}(\R^{N})} {\big)} \|\psi_{-\ell}\|_{L^{\mt{r}}_{v}(\R^{N})}\\
&\leq C \|f_{\ell}\|_{L^{\mt{p}}_{v}(\R^{N})}\|g_{\ell}\|_{L^{\mt{q}}_{v}(\R^{N})} \|\psi_{-\ell}\|_{L^{\mt{r}}_{v}(\R^{N})},
\end{split}
\end{equation}
thus we conclude~\eqref{W-convolution} by duality.

The proof of estimate~\eqref{W-convolution}  is a easy consequence of above argument which can be done 
by revising~\eqref{dual-w-est}. Let $1/a_{1}=1/\mt{p}_{m}+1/m$ and $1/a_{2}=1/\mt{q}_{m}+1/m$ and note that
we then have 
\[
\frac{1}{a_{1}}+\frac{1}{\mt{q}_{m}}=1+\frac{\gamma}{N}+\frac{1}{\tau_{m}}\;\;\;{\rm or}\;\;\;
\frac{1}{\mt{p}_{m}}+\frac{1}{a_{2}}=1+\frac{\gamma}{N}+\frac{1}{\tau_{m}}
\]
which is exactly~\eqref{scaling-relation}. Parallel to~\eqref{dual-w-est}, we have  
\begin{equation}\label{dual-w-est-2}
\begin{split}
&{\Big |}\int  Q^{+}(f,g)(v)\; \psi (v) dv {\Big |}\\
&\leq C{\big (} \|f_{\ell}\|_{L^{\mt{p}_{m}}_{v}(\R^{N})}\|g\|_{L^{a_{2}}_{v}(\R^{N})} + 
\|f\|_{L^{a_{1}}_{v}(\R^{N})}\|g_{\ell}\|_{L^{\mt{q}_{m}}_{v}(\R^{N})} {\big)} \|\psi_{-\ell}\|_{L^{\mt{r}_{m}}_{v}(\R^{N})}\\
&\leq C {\big (} \|f_{\ell}\|_{L^{\mt{p}_{m}}_{v}(\R^{N})}\|g_{\ell}\|_{L^{\mt{q}_{m}}_{v}(\R^{N})} \|\lr{v}^{-\ell}\|_{L^{m}_{v}(\R^{N})}\\
&\hskip3cm+ \|f\|_{L^{\mt{p}_{m}}_{v}(\R^{N})}\|\lr{v}^{-\ell}\|_{L^{m}_{v}(\R^{N})}\|g_{\ell}\|_{L^{\mt{q}_{m}}_{v}(\R^{N})} {\big)} \|\psi_{-\ell}\|_{L^{\mt{r}_{m}}_{v}(\R^{N})}\\
&\leq C \|f_{\ell}\|_{L^{\mt{p}_{m}}_{v}(\R^{N})}\|g_{\ell}\|_{L^{\mt{q}_{m}}_{v}(\R^{N})} \|\psi_{-\ell}\|_{L^{\mt{r}_{m}}_{v}(\R^{N})},
\end{split}
\end{equation}
where we used the condition $\ell m>N$ in the last inequality. By duality, we conclude~\eqref{W-convolution-2}. 
\end{proof}

To  prove the weighted estimate for the loss term.  We begin with the following.
\begin{prop}[The Hardy-Littlewood-Sobolev inequality]
If $x,y\in\R^N$, $1<p,q<\infty$, $-N<\gamma<0$ and $\frac{1}{p}+\frac{1}{q}+\frac{-\gamma}{N}=2$, 
then we have 
\[
{\Big |} \iint f(x)|x-y|^{\gamma} g(y) dxdy {\Big |}\leq C_{N,\gamma,p}\|f\|_{L^p}\|g\|_{L^q}
\]
\end{prop}

\begin{prop}\label{P:Loss}
Assume $B(v-v_*,\omega)$ defined in~\eqref{D:kernel} satisfies $-N<\gamma<0$
and~\eqref{D:Grad}. If $\ell>N/m$ and $1<\mt{p}_{m}, \mt{q}_{m}, m, \mt{r}_{m} <\infty$ satisfy
\begin{equation}\label{minus-size-condition-2}
\frac{1}{\mt{p}_{m}}+\frac{1}{\mt{r_{m}}'}<1\;,\;\frac{1}{\mt{q}_{m}}+\frac{1}{m}<1,
\end{equation}
and
\begin{equation}\label{scaling-relation-w2}
\frac{1}{\mt{p}_{m}}+\frac{1}{\mt{q}_{m}}+\frac{1}{m}=1+\frac{\gamma}{N}+\frac{1}{\mt{r}_{m}}, 
\end{equation}
then we have  
\begin{equation}\label{Lose-ineq}
\|\lr{v}^{\ell} Q^{-}(f,g)\|_{L^{\mt{r}_{m}}}\leq C(\mt{p}_{m},\ell) \|\lr{v}^{\ell} f\|_{L^{\mt{p}_{m}} }\|\lr{v}^{\ell} g\|_{L^{\mt{q}_{m}}}.
\end{equation}
Here we note that the first inequality of~\eqref{minus-size-condition-2} is equivalent to $1/\mt{p}_{m}<1/\mt{r}_{m}$. 
\end{prop}
\begin{proof}
Assume  $1<\mt{p}_{m} ,\mt{q}_{m} ,m, \mt{r}_{m} <\infty$ satisfy~\eqref{minus-size-condition} and~\eqref{scaling-relation}.
Let $1<a_{1},a_{2}<\infty$ be defined by
\begin{equation}\label{a-1-2}
\frac{1}{a_1}:=\frac{1}{\mt{p}_{m}}+\frac{1}{\mt{r_{m}}'},\;\frac{1}{a_2}:=\frac{1}{\mt{q}_{m}}+\frac{1}{m}.
\end{equation}
Then we have $1<a_1,a_2<\infty$ and
\begin{equation}
\frac{1}{a_1}+\frac{1}{a_2}=2+\frac{\gamma}{N}.
\end{equation}
By Hardy-Littlewood-Sobolev inequality and H\"{o}lder inequality, we have
\begin{equation}\label{Loss-HLS}
\begin{split}
{\Big |}\lr{Q^-(f,g),h} {\Big |}&={\Big |}\iint\int  f(v)g(v_*) h(v) B(v-v_*,\omega) 
d\omega dv_* dv {\Big |}\\
&={\Big |}C \iint  f(v)g(v_*) h(v) |v-v_*|^{\gamma} dv_*dv {\Big |}\\
&\leq C \|f\cdot h\|_{L^{a_1}}\|g\|_{L^{a_2}}\\
&\leq  C\|\lr{v}^{\ell}f\|_{L^{\mt{p}_{m}}}\|\lr{v}^{-\ell} h\|_{L^{\mt{r_{m}}'}}\|\lr{v}^{\ell} g\|_{L^{\mt{q}_{m}}}\|\lr{v}^{-\ell}\|_{L^{m}}\\
&\leq C\|\lr{v}^{\ell}f\|_{L^{\mt{p}_{m}}}\|\lr{v}^{-\ell} h\|_{L^{\mt{r_{m}}'}}\|\lr{v}^{\ell} g\|_{L^{\mt{q}_{m}}},
\end{split}
\end{equation}
where the last inequality holds when $\ell>N/m$. 
Then we conclude that ~\eqref{Lose-ineq} holds.
\end{proof}
\begin{rem}\label{single-weight}
Our proof also includes the estimate
\[
\|\lr{v}^{\ell} Q^{-}(f,g)\|_{L^{\mt{r}_{m}}}\leq C(\mt{p}_{m},\ell) \|\lr{v}^{\ell} f\|_{L^{\mt{p}_{m}} }
\|g\|_{L^{a_{2}}}.
\]
where $1/a_{2}=1/\mt{q}_{m}+1/m$ and
\[
\frac{1}{\mt{p}_{m}}+\frac{1}{a_{2}}=1+\frac{\gamma}{N}+\frac{1}{\mt{r}_{m}}.
\]
\end{rem}

We also need to build the weighted Strichartz estimates.   
We consider the weight $\lr{v}^{\ell},\;\ell\in\mathbb{R}$ as the multiplication operator. 
Using the  notations of~\eqref{D:solution-maps}, we note that the following communication relations hold:
\begin{equation}
 \lr{v}^{\ell} U(t)u_{0}=U(t) \lr{v}^{\ell} u_{0},\; \lr{v}^{\ell} W(t)F=W(t) \lr{v}^{\ell} F.
\end{equation}
Combining above facts with Proposition~\ref{Strichartz-est}, we have the following result.
\begin{cor}\label{W-Strichartz}
Let $u$ satisfy the kinetic transport equation~\eqref{E:KT}.  The estimate
\begin{equation}\label{E:Strichartz-w}
\begin{split}
& \|  \lr{v}^{\ell} u \|_{L^q_tL^r_xL^p_v}\\
&{\hskip 1cm}\leq C(q,r,p,N){\big (} \|   \lr{v}^{\ell} u_0\|_{L^a_{x,v}} +\|   \lr{v}^{\ell} F\|_{L^{\tilde{q}'}_tL^{\tilde{r}'}_xL^{\tilde{p}'}_v}  {\big)}
\end{split}
\end{equation}
holds for all $u_0\in L^a_{t,x}$ and all $F\in {L^{\tilde{q}'}_tL^{\tilde{r}'}_xL^{\tilde{p}'}_v} $ if and only if 
$(q,r,p)$ and $(\tilde{q},\tilde{r},\tilde{p})$ are two KT-admissible exponents triplets and $a=$HM$(p,r)=$HM$(\tilde{p}',\tilde{r}')$ with the exception of $(q,r,p)$ being an endpoint triplet.
\end{cor}

\subsection{Gain term only equation for $N=3$ and $-1<\gamma\leq 0$}

\begin{prop}\label{Thm-Gain-2} 
Let $N=3$ and collision kernel $B$ defined in~\eqref{D:kernel} satisfies~\eqref{D:Grad} and $-1<\gamma \leq 0$. 
Let  $\ellg=(1+\gamma)^+<3/2$. There exists a small number $\eta>0$ such that if the initial data $f_0$ is in the set
\[B^{\ellg}_{\eta}=\{f_0\in L^3_{x,v} (\mathbb{R}^3\times
\mathbb{R}^3):  \| \lr{v}^{\ellg} f_0\|_{L^3_{x,v}}<\eta\}\subset L^3_{x,v},
\] 
there exists a globally unique  mild solution 
\[
\lr{v}^{\ellg} f_+\in C([0,\infty),L^3_{x,v})\cap
L^{\mbq}([0,\infty],L^{\mbr}_xL^{\mbp}_v) 
\]
where the triple $(\mbq,\mbr,\mbp)$ lies in the set
\begin{equation}\label{solvable-g-2}
\{ (\mbq,\mbr,\mbp) | \;\frac{1}{\mbq}=\frac{3}{\mbp}-1\;,\;\frac{1}{\mbr}=\frac{2}{3}-\frac{1}{\mbp}\;,\;
\frac{1}{3}<\frac{1}{\mbp}<\frac{4}{9}\}. 
\end{equation} 
The solution map $\lr{v}^{\ellg} f_0\in B^{\ellg}_{\eta}\rightarrow \lr{v}^{\ellg} f_+\in L^{\mbq}_tL^{\mbr}_xL^{\mbp}_v$ 
is Lipschitz continuous and the solution $\lr{v}^{\ellg} f_+$ scatters with respect to the kinetic 
transport operator in $L^3_{x,v}$.  
\end{prop}
\begin{proof}
The proof can be done by exactly the same argument as that for Proposition~\ref{Thm-Gain} except that 
the estimates there need to be replaced by weighted ones. Parallel to~\eqref{contraction}, we claim that the 
following key estimates hold:
\begin{equation}\label{contraction-2}
\begin{split}
\|S (\lr{v}^{\ellg} f_+)\|_{L^{\mbq}_tL^{\mbr}_xL^{\mbp}_v}
& \leq C_{0}\| \lr{v}^{\ellg} f_{0}\|_{L^N_{x,v}}+
C_{1}\| \lr{v}^{\ellg}Q^+(f_+,f_+)\|_{L^{\tmbq'}_tL^{\tmbr'}_xL^{\tmbp'}_v} \\
&  \leq C_{0}\| \lr{v}^{\ellg} f_{0}\|_{L^N_{x,v}}+C_{2}\| \lr{v}^{\ellg} f_+\|^2_{L^{\mbq}_tL^{\mbr}_xL^{\mbp}_v}.
\end{split}
\end{equation}
where we choose the same triplets as Proposition~\ref{Thm-Gain} (see also Definition~\ref{solvable-tri}).
Then the first inequality follows from Corollary~\ref{W-Strichartz}. To show the second inequality is 
equal to show
\begin{equation}\label{desire-est-sol-2}
\| \lr{v}^{\ellg}Q^+(f_+,f_+)\|_{L^{\tmbq'}_tL^{\tmbr'}_xL^{\tmbp'}_v}  \leq C 
 \| \lr{v}^{\ellg} f_+\|^2_{L^{\mbq}_tL^{\mbr}_xL^{\mbp}_v}, 
\end{equation}
which is parallel to~\eqref{desire-est-sol}. Please note that the only change occurs at $v$ variable and 
at which we should verify.  

When $\gamma=2-N=-1$, the relation~\eqref{scaling-relation-0} and the proof of~\eqref{desire-est-sol} require
\begin{equation}\label{v-scaling}
\frac{1}{\mbp}+\frac{1}{\mbp}=\frac{2}{3}+\frac{1}{\tmbp'}.
\end{equation}

When $-1<\gamma\leq 0$,  we can rewrite~\eqref{v-scaling} as  
\begin{equation}\label{v-scaling-w}
\frac{1}{\mbp}+\frac{1}{\mbp}+\frac{1+\gamma}{3}=1+\frac{\gamma}{3}+\frac{1}{\tmbp'}
\end{equation}
and note that this is the form of~\eqref{scaling-relation-w} with $1/m=(1+\gamma)/3\leq 1/3$. Since $1/3<1/\mbp<4/9$, 
the relation~\eqref{minus-size-condition}  is satisfied. Let  $\ellg=(3/m)^+=(1+\gamma)^+>N/m=(1+\gamma)$, 
then we have~\eqref{W-convolution-2}, i.e.,~\eqref{desire-est-sol-2}. 
\end{proof}

Also we need a weighted version of  Proposition~\ref{P-p2-est}. 
\begin{prop}\label{P-p2-est-w}
Let $N=3$ and $a_2=15/8$. Under the same assumption as Proposition~\ref{Thm-Gain-2}, if we  further assume 
$\|\lr{v}^{\ellg} f_0\|_{L^{a_2}_{x,v}}<\infty$, then solution $f_+$ in Proposition~\ref{Thm-Gain-2} also satisfies
\begin{equation}
\|\lr{v}^{\ellg}f_+\|^2_{L^{q_2}_t L^{r_2}_x L^{p_2}_v}<\infty
\end{equation} 
where $(1/q_2,1/r_2,1/p_2)=(1/2,\;(12+\gamma)/30,\;(20-\gamma)/30)$ is a KT-admissible triplet with $1/p_2+1/r_2=2/a_2$. 
\end{prop}
\begin{proof}
Follow the idea of Proposition~\ref{P-p2-est}, it suffices to show that 
\begin{equation}\label{Gain-p2-w}
\|\lr{v}^{\ellg} Q^+(f_+,f_+)\|_{L^{\tilde{q}'_{2}}_tL^{\tilde{r}'_{2}}_xL^{\tilde{p}'_{2}}_v} \leq C
\|\lr{v}^{\ellg} f_+\|_{L^{\mbq}_tL^{\mbr}_xL^{\mbp}_v}\|\lr{v}^{\ellg} f_+\|_{L^{q_{2}}_tL^{r_{2}}_xL^{p_{2}}_v}.
\end{equation}
Using the same trick of rewriting~\eqref{v-scaling} as~\eqref{v-scaling-w}, we rewrite~\eqref{uni-v}, i.e.,
${1}/{\mbp}+{1}/{p_{2}}={2}/{3}+{1}/{\tilde{p}'_{2}}$ as
\begin{equation}
\frac{1}{\mbp}+\frac{1}{p_{2}}+\frac{1+\gamma}{3}=1+\frac{\gamma}{3}+\frac{1}{\tilde{p}'_{2}}
\end{equation}
and apply~\eqref{W-convolution-2} of Proposition~\ref{P:Convolution-w}. 
\end{proof}

Using again the argument of Proposition~\ref{P-p2-est-w} and Proposition~\ref{P:Loss} and the result 
of Proposition~\ref{P:Convolution-w}, we have the following
  weighted version of Corollary~\ref{C-L-p2}.
\begin{cor}\label{C-L-p2-w} 
Use the same notations as Proposition~\ref{P-p2-est-w} but $\gamma\neq 0$.
Suppose $\lr{v}^{\ellg} f_{1}\in L^{\mbq}_tL^{\mbr}_xL^{\mbp}_v$ and $\lr{v}^{\ellg} f_{2}\in L^{q_{2}}_tL^{r_{2}}_xL^{p_{2}}_v$. Then 
$\lr{v}^{\ellg} Q^{\pm}(f_{1},f_{2})\in L^{\tilde{q}'_{2}}_tL^{\tilde{r}'_{2}}_xL^{\tilde{p}'_{2}}_v$ and 
\[
\begin{split}
& \|\lr{v}^{\ellg} Q^-(f_1,f_2)\|_{L^{\tilde{q}'_{2}}_tL^{\tilde{r}'_{2}}_xL^{\tilde{p}'_{2}}_v} \leq C
\|\lr{v}^{\ellg} f_1\|_{L^{\mbq}_tL^{\mbr}_xL^{\mbp}_v}\|\lr{v}^{\ellg} f_2\|_{L^{q_{2}}_tL^{r_{2}}_xL^{p_{2}}_v},\\
&\|\lr{v}^{\ellg} Q^+(f_1,f_2)\|_{L^{\tilde{q}'_{2}}_tL^{\tilde{r}'_{2}}_xL^{\tilde{p}'_{2}}_v} \leq C
\|\lr{v}^{\ellg} f_1\|_{L^{\mbq}_tL^{\mbr}_xL^{\mbp}_v}\|\lr{v}^{\ellg} f_2\|_{L^{q_{2}}_tL^{r_{2}}_xL^{p_{2}}_v},\\
&\|\lr{v}^{\ellg} Q^+(f_1,f_2)\|_{L^{\tilde{q}'_{2}}_tL^{\tilde{r}'_{2}}_xL^{\tilde{p}'_{2}}_v} \leq C
\|\lr{v}^{\ellg} f_2\|_{L^{\mbq}_tL^{\mbr}_xL^{\mbp}_v}\|\lr{v}^{\ellg} f_1\|_{L^{q_{2}}_tL^{r_{2}}_xL^{p_{2}}_v},
\end{split}
\]
\end{cor}

\subsection{Proof of Theorem~\ref{result2}}

\begin{proof}[Proof of Theorem~\ref{result2}]
As we did for the case $\gamma=-1$, we use the solution $f_+$ in Proposition~\ref{Thm-Gain-2} to construct
the beginning condition of the the Kaniel-Shinbrot iteration, i.e. Let $g_1=f_+$ and $h_1\equiv 0$. Then the 
system~\eqref{h-g-2} gives $g_2=f_+=g_1$. By Corollary~\ref{C-L-p2-w} and the argument after~\eqref{p2boundf+},
we have $\lr{v}^{\ellg}L(g_{1})\in L^{\mbq_{2}}_{t}L^{\mbr_{2}}_{x}L^{\mbp_{2}}_{v}$ where $(1/\mbq_{2})'=1/\tilde{q}_{2}+1/\mbq$
, $(1/\mbr_{2})'=1/\tilde{r}_{2}+1/\mbr$ and $(1/\mbp_{2})'=1/\tilde{p}_{2}+1/\mbp$ ($(q_2,r_2,p_2)$ is given by 
Proposition~\ref{P-p2-est-w}). Therefore $L(g_{1})$ is pointwisely a.e. well-defined. Then we can compute $h_2$ by~\eqref{h-g-2}
and have the beginning condition $0\leq h_1\leq h_2\leq g_2\leq g_1$.  Hence the limit  functions $g,h$ of iteration exist
and we have 
\begin{equation}\label{g-h-finite-2}
\begin{split}
&\lr{v}^{\ellg} g,\;\lr{v}^{\ellg} h \in  L^{\mbq}([0,\infty],L^{\mbr}_xL^{\mbp}_v), \\
& \lr{v}^{\ellg} g,\;\lr{v}^{\ellg} h \in  L^{q_{2}}([0,\infty],L^{r_{2}}_xL^{p_{2}}_v), \\
& \lr{v}^{\ellg} Q^{+}(g,g), \;\lr{v}^{\ellg} Q^{+}(h,h)\in L^{\tmbq'}_tL^{\tmbr'}_xL^{\tmbp'}_{v},\\
& \lr{v}^{\ellg} Q^{\pm}(g,g),\;\lr{v}^{\ellg} Q^{\pm}(h,h)\in L^{\tilde{q}'_{2}}_tL^{\tilde{r}'_{2}}_xL^{\tilde{p}'_{2}}_{v}.
\end{split}
\end{equation}

Replacing the estimates from Lemma~\ref{w-equation} to the end of Section 3 by their weighted version, we
see that the remaining part of the proof follows. For example,~\eqref{w-est} is replaced by 
\begin{equation}\label{w-est-2}
\begin{split}
&\|\lr{v}^{\ellg} w\|_{L^{q_{2}}([t_{0},s],L^{r_{2}}_xL^{p_{2}}_v)}\\
&\leq C( \| \lr{v}^{\ellg} g\|_{L^{\mbq}([t_{0},s],L^{\mbr}_xL^{\mbp}_v)}+ \|\lr{v}^{\ellg} h\|_{L^{\mbq}([t_{0},s],L^{\mbr}_xL^{\mbp}_v)})
\| \lr{v}^{\ellg} w\|_{L^{q_{2}}(([t_{0},s],L^{r_{2}}_xL^{p_{2}}_v)}\\
&:=C(g,h,s) \|\lr{v}^{\ellg} w\|_{L^{q_{2}}([t_{0},s],L^{r_{2}}_xL^{p_{2}}_v)},
\end{split}
\end{equation}
and the others are similar. 
\end{proof}

\noindent{\bf Acknowledgments.} L.-B. He is   supported by NSF of CHINA under Grants 11771236 and 12141102. 
J.-C. Jiang was supported in part by 
 National Sci-Tech Grant MOST 109-2115-M-007-002-MY3.

\end{document}